\documentclass[leqno,11pt]{amsart}
\usepackage{amssymb,amsmath,amsfonts,amsthm,mathtools,latexsym,ifthen,bezier}
\usepackage{graphicx}
\usepackage{amsxtra}
\usepackage{xspace}
\usepackage{url}
\usepackage[english]{babel} 

\usepackage{enumitem}
\usepackage[margin=1.3in]{geometry}
\usepackage{setspace}
\usepackage[english]{babel}
\usepackage{bussproofs}
\usepackage{csquotes}
\usepackage{mathtools}
\usepackage{mathrsfs}
\usepackage{latexsym}
\usepackage{enumitem}
\usepackage{amsthm}
\usepackage{amssymb}
\DeclareMathAlphabet{\mathpzc}{OT1}{pzc}{m}{it}
\usepackage[latin1]{inputenc}
\usepackage{afterpage}

\usepackage{bbm}
\usepackage{tikz} 

\newtheorem{theorem}{Theorem}[section]

\newtheorem{lemma}[theorem]{Lemma}
\newtheorem{proposition}[theorem]{Proposition}
\newtheorem{corollary}[theorem]{Corollary}

\newtheorem{fact}[theorem]{Fact}
\newtheorem{claim}[theorem]{Claim}

\theoremstyle{definition}
\newtheorem{definition}[theorem]{Definition}

\theoremstyle{remark}
\newtheorem{remark}{Remark}

\newtheorem{question}{Question}

\def\hook{\upharpoonright}
\def\forces{\Vdash}

\newfont{\ssi}{cmssi12 at 12pt}

\newenvironment{ea*}{\begin{eqnarray*}}{\end{eqnarray*}}

\setbox0=\hbox{$\longrightarrow$}

\newcommand{\calA}{\mathcal{A}}

\renewcommand{\phi}{\varphi}

\newcommand{\ZFC}{\ensuremath{\mathsf{ZFC}}\xspace}

\def\<#1>{\langle#1\rangle}

\renewcommand{\P}{{\mathord{\mathbb P}}}
\newcommand{\Q}{{\mathord{\mathbb Q}}}

\newcommand{\MA}{\ensuremath{\mathsf{MA}}}

\newcommand{\MP}{\ensuremath{\mathsf{MP}}}

\newcommand{\ColNothing}{\mathrm{Col}}
\newcommand{\Col}[1]{\ColNothing(#1)}

\newcommand{\MPColNothing}[1]{\MP_{\Col{\dot{\kappa}}}}

\newcommand{\CH}{\ensuremath{\mathsf{CH}}\xspace}

\newcommand{\calP}{\mathcal{P}}

\def\hook{\upharpoonright}
\def\forces{\Vdash}

\def\Me{\mathcal M}
\def\Null{\mathcal N}
\def\ZFC{\mathsf{ZFC}}

\def\MA{\mathsf{MA}}

\def\baire{\omega^\omega}

\def\cantor{2^\omega}

\def\mfb{\mathfrak b}
\def \mfd{\mathfrak{d}}
\def\mfa{\mathfrak{a}}

\def\CH{\mathsf{CH}}

\def\calP{\mathcal P}
\def\calE{\mathcal E}
\def\calA{\mathcal A}
\def\calI{\mathcal I}

\title{Tight Eventually Different Families}
\author[Fischer]{Vera Fischer}
\address[V. ~Fischer]{Institut f\"{u}r Mathematik, Kurt G\"odel Research Center, Universit\"{a}t Wien, Kolingasse 14-16, 1090 Wien, AUSTRIA}
\email{vera.fischer@univie.ac.at}

\author[Switzer]{Corey Bacal Switzer}
\address[C.~B.~Switzer]{Institut f\"{u}r Mathematik, Kurt G\"odel Research Center, Universit\"{a}t Wien, Kolingasse 14-16, 1090 Wien, AUSTRIA}
\email{corey.bacal.switzer@univie.ac.at}

\thanks{\emph{Acknowledgements:} The authors would like to thank the
Austrian Science Fund (FWF) for the generous support through grant number Y1012-N35.}

\date{}


\begin{document}
\maketitle

\begin{abstract}
Generalizing the notion of a tight almost disjoint family, we introduce the notions of a {\em tight eventually different} family of functions in Baire space and a {\em tight eventually different set of permutations} of $\omega$. Such sets strengthen maximality, exist under $\MA (\sigma {\rm -centered})$ and come with a properness preservation theorem. The notion of tightness also generalizes earlier work on the forcing indestructibility of maximality of families of functions. As a result we compute the cardinals $\mfa_e$ and $\mfa_p$ in many known models  by giving explicit witnesses and therefore obtain the consistency of several constellations of cardinal characteristics of the continuum including $\mfa_e = \mfa_p = \mfd < \mfa_T$, $\mfa_e = \mfa_p < \mfd = \mfa_T$, $\mfa_e = \mfa_p =\mathfrak{i} < \mathfrak{u}$ and $\mfa_e=\mfa_p = \mfa < non(\Null) = cof(\Null)$. We also show that there are $\Pi^1_1$ tight eventually different families and tight eventually different sets of permutations in $L$ thus obtaining the above inequalities alongside $\Pi^1_1$ witnesses for $\mfa_e = \mfa_p = \aleph_1$.

Moreover, we prove that tight eventually different families are Cohen indestructible and are never analytic.
\end{abstract}

\section{Introduction}

Two infinite subsets of $\omega$, $A, B \in [\omega]^\omega$ are {\em almost disjoint} if $A \cap B$ is finite. A family $\calA \subseteq [\omega]^{\omega}$ is {\em almost disjoint} if its elements are pairwise almost disjoint. Such a family is {\em maximal} or MAD if it is almost disjoint but not properly included in any other almost disjoint family. The least size of an infinite MAD family is denoted by $\mfa$ and is one of the most studied cardinal invariants of the continuum. See \cite[Section 8]{BlassHB} for more on $\mfa$. 

The cardinal $\mfa$ has many relatives that have also been studied in the literature. These essentially come in two forms: either we can replace ``finite" in the definition of almost disjoint with belonging to another ideal or we can insist that the MAD family come equipped with additional structure. Some examples of this are as follows:

\begin{enumerate}
\item
$\mfa_e$, the least size of a maximal eventually different family of functions in $\baire$.
\item
$\mfa_p$, the least size of a maximal eventually different set of permutations of $\omega$.
\item
$\mfa_T$, the least size of a maximal almost disjoint family of finitely branching trees $T \subseteq \omega^{<\omega}$. This cardinal is equivalent to the least size of an uncountable partition of any perfect Polish space into compact sets.
\item
$\mfa_g$, the least size of a maximal cofinitary group\footnote{Recall that a subgroup $G \leq S(\omega)$ is {\em cofinitary} if every non-identity element has only finitely many fixed points. It's not hard to see that any cofinitary group consists of eventually different permutations.}.
\end{enumerate}

Here, recall that a set $\mathcal E \subseteq \baire$ is {\em eventually different} if for all distinct $f, g\in \calE$ there is a $k < \omega$ so that for all $n > k$ $f(n) \neq g(n)$. Denote this situation by $f \neq^* g$. Dually if $f \neq^* g$ fails then we say that $f$ and $g$ are {\em infinitely often equal}, denoted $f =^\infty g$. An eventually different family is {\em maximal} if it is not contained in any strictly larger such family. The same can be said in the case of $\mfa_p$ with the additional stipulation that every element is a bijection of $\omega$ with itself. For more examples of relatives of $\mfa$ see \cite{restrictedMADfamilies, ADTop}.

In general the relationship between these cardinals remains murky. It's known that $\mfa$ can be consistently less than all of them, but it's not known if the reverse inequality is consistent for any of them. It remains open if $\mfa_e = \mfa_p$ is provable in $\ZFC$. In this paper we contribute to the project of separating and understanding these invariants. We provide several models of $\mfa_e = \mfa_p = \aleph_1$ while allowing other related cardinals to be $\aleph_2$ in a controlled way. These results are obtained using new strengthenings of maximality for eventually different families of functions and permutations respectively.

A fruitful strengthening of maximality for almost disjoint families in $[\omega]^\omega$ if given by the notion of a {\em tight} MAD family. An almost disjoint family $\calA$ is {\em tight} if given any countable set $\{B_n \; | \; n<\omega\}$ of $\mathcal I(\mathcal A)^+$ sets there is a single $C \in \mathcal I(\mathcal A)$ so that $B_n \cap C$ is infinite for all $n < \omega$. Tight MAD families exist under $\mfb = 2^{\aleph_0}$ (see \cite{orderingMAD}), are Cohen indestructible (\cite[Corollary 3.2]{orderingMAD}) and, under certain conditions are preserved by countable support iterations of proper forcing notions (see~\cite{restrictedMADfamilies}).  

On the other hand, the preservation of the maximality of maximal eventually different families under forcing iterations are not that well understood. 
Using the parametrized $\diamondsuit$ principles of \cite{DiamondPrinc}, Kastermans and Zhang show that $\mfa_g = \mfa_e$ in the Miller model (see~\cite{KastermansZhang06}). Their construction can be easily modified to establish the analogous result for maximal families of eventually different functions. A more recent study in the area is the work of the first author with D. Schrittesser, see~\cite{SacksMedf}, where they establish the existence (in the Constructible Universe $L$) of a co-analytic Sacks-indestructible maximal eventually different family. Note that the techniques of ~\cite{KastermansZhang06} and~\cite{SacksMedf} are completely different, and moreover completely different than the techniques developed in the current article. 

The results of the current paper improve and generalize the above, as we produce a uniform framework for the preservation of maximal families of eventually different functions, a framework, which applies to a long list of partial orders, including Sacks forcing, Miller rational perfect tree forcing, partition forcing, infinitely often equal forcing, Shelah's poset 
for destroying the maximality of a given maximal ideal, as well as their countable support iterations.

For an eventually different family $\calE$, we denote by $\calI_T(\calE)$ the family of all trees $T\subseteq{\omega^{<\omega}}$ so that there is $t\in T$ and a finite $X\in [\calE]^{<\omega}$ such that $\bigcup T_t\subseteq \bigcup X$, where $T_t=\{s\in T: s\subseteq t\hbox{ or }t\subseteq s\}$. We refer to this family, as the tree ideal generated by $\calE$ (see Definition~\ref{def.treeideal}). Then $\calI_T(\calE)^+$ is defined as the collection of all trees $T\subseteq\omega^{<\omega}$ with the property that for each $t\in T$, $\bigcup T_t$ is not almost covered by finitely many functions from $\calE$. Finally, we say that an eventually different family $\calE$ is tight, if for every $\{T_n\}_{n\in\omega}\subseteq \calI_T^+(\calE)$ there is a single $g\in\calE$ with the property that $\forall n\in\omega\forall t\in T_n$ there is a branch $g_{n,t}$ of $T_n$ such that $t\subseteq g_{n,t}$ and for infinitely many $m$, $g_{n,t}(m)=g(m)$ (see Definition~\ref{def.tight}).

Note that the notion of tightness for an eventually different family is a natural strengthening of maximality for such families (see Proposition~\ref{if.tight.then.max}) as well as a strengthening of $\omega$-maximality (see ~\cite{SFLZ}).  As shown in Proposition~\ref{MA.tight}, tight eventually different families exist under $\hbox{\textsf{MA}}(\sigma\hbox{-centered})$ and so in particular under \textsf{CH}. In Section 4, we introduce strong preservation of a tight mad families and show that the countable support iteration of proper forcing notions, each iterand of which strongly preserves a tight mad family, also preserve the tightness of an eventually different family. Equipped with the above techniques, we establish:

\begin{theorem}\label{MainTheorem}
The following inequalities are all consistent and in each case $\mfa_e = \mfa_p = \aleph_1$ is witnessed by a tight eventually different family and a tight eventually different set of permutations respectively.
\begin{enumerate}
\item
$\mfa = \mfa_e = \mfa_p < \mfd = \mathfrak{u} =  \mfa_T = 2^{\aleph_0}$
\item
$\mfa= \mfa_e = \mfa_p = \mfd < \mfa_T = 2^{\aleph_0}$
\item
$\mfa = \mfa_e = \mfa_p = \mfd = \mathfrak{u} < non(\Null) = cof(\Null) = 2^{\aleph_0}$. 
\item
$\mfa = \mfa_e = \mfa_p = \mathfrak{i} = cof(\Null) < \mathfrak{u}$.
\end{enumerate}
Moreover, if we work over the constructible universe, we can provide co-analytic witnesses of cardinality $\aleph_1$ to each of $\mathfrak{a},\mathfrak{a}_e,\mathfrak{a}_p,\mathfrak{i},\mathfrak{u}$ in the above inequalities. 
\label{mainthm}
\end{theorem}


The rest of this paper is organized as follows. In the remaining part of the introduction we record some preliminaries for later use. In Section 2 we introduce tight eventually different families, which is of central focus for the entire paper. In Section 3 we prove that tight eventually different families are Cohen indestructible and never analytic. Section 4 contain our preservation results for such families. In Section 5, we observe that the discussion of tight eventually different families applies {\em mutatis mutandis} to its generalization for eventually different sets of permutations. Next we turn to applications. In Sections 6 through 9 we prove that several well known forcing notions strongly preserve the tightness of eventually different families of functions and permutations, and conclude the consistency of the inequalities from the above Theorem. In Section 10 we study the definability properties of tight families of functions and provide co-analytic witnesses for the results described in~\ref{MainTheorem}. We conclude the paper with open questions for future research.

\subsection{Preliminaries}
Our first preliminary involves recalling the {\em ideal associated to an} almost disjoint family. Recall that given an almost disjoint family $\calA$ the {\em ideal associated to }$\calA$ is the set $\mathcal I(\calA)$ consisting of all $B \in [\omega]^\omega$ for which there is a finite set $\{A_0, ..., A_{n-1}\}\subseteq \calA$ so that $B \subseteq^* \bigcup_{k < n} A_k$. This ideal is important when considering how MAD families persist (or not) to forcing extensions. Often this takes the following form. If $\P$ is some forcing notion and $\dot{X}$ is a $\P$-name for an element of $\calI(\calA)^+$ then its {\em outer hull}, the set of all $m$ for which some condition forces $\check{m} \in \dot{X}$ is an element of $\calI (\calA)^+$ in $V$. Many arguments involving MAD families in forcing extensions consider this set at key points. One of the main technical observations of this paper involves the introduction of a similar ideal-like family associated to an eventually different family. However, the naive generalization, thinking of an eventually different family as collection of subsets of $P(\omega^2)$ does not work and instead we consider an ideal-like family (which is not actually an ideal) generated by the eventually different family consisting of trees $T \subseteq \omega^{<\omega}$. This is described in detail in Section 2, see also Remark 1.

We conclude this introduction by recording some of what was known already concerning $\mfa$, $\mfa_e$, $\mfa_p$, $\mfa_T$, $\mfb$ and $\mfd$. To start we recall the known provable inequalities that will be used throughout the paper. 
\begin{fact}
The following inequalities are provable in $\ZFC$.
\begin{enumerate}
\item
$non(\Me) \leq \mfa_e$
\item (\cite[Theorem 2.2]{nonmandag}
$non(\Me) \leq \mfa_p$
\item (\cite[Proposition 8.4]{BlassHB})
$\mfb \leq \mfa$
\item (\cite[Theorem 2.5]{partitionnumbers})
$\mfd \leq \mfa_T$
\end{enumerate}
\end{fact}

\begin{proof}
All of these have been proved in the literature (see the citations) with the exception of $non(\Me) \leq \mfa_e$. This inequality however follows from the well known fact that $non(\Me)$ is equal to the bounding number for $\neq^*$, i.e. $non (\Me)$ is the least size of a set of reals $A \subseteq \baire$ for which there is no $f \in A$ which is eventually different from all $g \in A$, see \cite[Theorem 2.4.7]{BarJu95}. Indeed, given this, suppose $A \subseteq \baire$ is eventually different and $|A| < non(\Me)$. Then, there must be a real $f \in \baire$ eventually different from every element of $A$, so $A$ is not maximal.
\end{proof}

Note that since $\mfb \leq non(\Me), \mfd$ the bounding number is actually a lower bound on all the relatives of $\mfa$ (see \cite{MADfamandNeighbors} for a discussion of this). It's also known that both $non(\Me)$ and $\mfd$ are independent of $\mfa$, and that $non(\Me)$ and $\mfa_T$ are independent (see below). Moreover $\mfd$ and $\mfa_e/\mfa_p$ are independent. In the case of $\mfa_p$ this was known, though perhaps never written down. Namely, $non(\Me) = 2^{\aleph_0}$ in the random model, and therefore $\mfa_e = \mfa_p = 2^{\aleph_0}$ while $\mfd = \aleph_1$. On the other hand Kastermans and Zhang proved in \cite{KastermansZhang06} that $\mfa_p = \aleph_1$ in the Miller model, where it's well known that $\mfd = \aleph_2$. Their proof uses the parametrized diamonds of \cite{DiamondPrinc}, and is completely different than the proof we give of this result in Section 5 below. Though it's not stated in their paper, Kastermans and Zhang's proof that $\mfa_p = \aleph_1$ in the Miller model can also be easily augmented to show that $\mfa_e = \aleph_1$ in the Miller model as well. In Section 5 we give a different proof that $\mfa_e = \aleph_1$ in the Miller model as well hence giving a new proof of the independence of $\mfd$ and $\mfa_e$. 

We finish this section by noting what is known about consistent inequalities between $\mfa_e$, $\mfa_p$, $\mfa_T$ and $\mfd$.

\begin{fact}
The inequality $\mfa_T = \mfd < \mfa_e = \mfa_p$ holds in the random model, in particular it is consistent.
\label{randommodel}
\end{fact}

\begin{proof}
That $\mfd =\aleph_1$ holds in the random model is well known, see, for example \cite[Model 7.6.8]{BarJu95}. Kunen and, independently, Stern showed that $\mfa_T = \aleph_1$ in the random model, see \cite[Theorem 5]{MillerCOV}. Meanwhile since $non(\mathcal M) \leq \mfa_e, \mfa_p$ and $non(\Me) = 2^{\aleph_0}$ in the random model, we have that $\mfa_e = \mfa_p = 2^{\aleph_0}$.
\end{proof}

It is also not hard to obtain the consistency of $\mfd < \mfa_T = \mfa_e = \mfa_p$.
\begin{proposition}
It is consistent that $\mfd = \aleph_1 < \mfa_T = \mfa_e = \mfa_p = \aleph_2$.
\label{randommodel2}
\end{proposition}

To describe this model we recall the partition forcing introduced in \cite{MillerCOV} (see also \cite{IndependentCompact} and \cite[Section 1.6]{restrictedMADfamilies}). This forcing will be discussed in more length in Section 7. Given an uncountable partition of size $\aleph_1$ of $\cantor$ into closed sets $\mathcal K = \{C_\alpha \; |\; \alpha < \omega_1\}$ let $\mathbb P(\mathcal K)$ be the set of all perfect trees $p$ so that for all $\alpha$ $[p] \cap C_\alpha$ is nowhere dense in $[p]$. The order is inclusion. This forcing was first investigated by Miller who showed that it is proper, has the Laver property and adds a real not in the evaluation of any $C_\alpha$ in the extension, thus killing the fact that $\mathcal K$ is a partition of $\cantor$. It follows that, via an appropriate bookkeeping device, an $\omega_2$-length countable support iteration of forcing notions of the form $\P(\mathcal K)$ will force $\mfa_T = \aleph_2$. Later Spinas, \cite[Lemma 2.7]{partitionnumbers} proved that this forcing notion is $\baire$-bounding. Black-boxing these facts the proof of Proposition \ref{randommodel2} is as follows.

\begin{proof}
Assume $\CH$ in the ground model. Define a countable support iteration $\langle \P_\alpha,  \dot{\Q}_\alpha \;  |\; \alpha < \omega_2\rangle$ so that for even $\alpha$ $\forces_\alpha$``$\dot{\Q}_\alpha$ is random forcing" and, using some appropriate bookkeeping device, for odd $\alpha$ $\forces_\alpha$ ``$\dot{\Q}_\alpha$ is partition forcing for some uncountable partition of $2^\omega$". In the resulting model $non(\Me) = 2^{\aleph_0}$ because of the random reals added and hence $\mfa_e = \mfa_p = 2^{\aleph_0}$. Similarly $\mfa_T = 2^{\aleph_0}$ because of the partition forcing iterands. Finally note that since all iterands are $\baire$-bounding the entire iteration is $\baire$-bounding and so $\mfd=\aleph_1$.
\end{proof}

Collectively, Fact \ref{randommodel}, Proposition \ref{randommodel2} and Theorems \ref{miller2} and \ref{partitionpreservation2} show no provable inequalities exist between $\mfa_e/\mfa_p$, $\mfd$ and $\mfa_T$ aside from $\mfd \leq \mfa_T$. This is stated more precisely as Corollary \ref{ind} below.

\section{Tight Eventually Different Families}
In this section we introduce a strong version of maximality for eventually different families which can be preserved by countable support iterations of proper forcing notions. Towards this we introduce some new terminology. Throughout this section fix an eventually different family $\calE$. Recall that a {\em tree} on $\omega$ is a set $T \subseteq \omega^{<\omega}$ which is closed downwards in the sense that if $s$ is an initial segment of $t \in T$ then $s \in T$. Throughout we will only restrict our attention to pruned trees i.e. trees $T$ in which every finite sequence $t \in T$ has a proper extension. Unless stated otherwise every tree in this article is assumed to be pruned. Note that if $X$ is a set of functions then $\bigcup X \subseteq \omega^2$ and similarly if $T\subseteq \omega^{<\omega}$ is a tree then $\bigcup T \subseteq \omega^2$. Therefore it makes sense to talk about a set of functions {\em covering a tree}, namely we say that a set of functions $X$ {\em covers} a tree $T$ if $\bigcup T \subseteq \bigcup X$ and $X$ {\em almost covers $T$} if $\bigcup T \subseteq^* \bigcup X$. Given a tree $T$ and a node $t \in T$ let $T_t = \{s \in T\; | \; s \subseteq t \; {\rm or} \; t \subseteq s\}$.

\begin{definition}\label{def.treeideal}[The Tree Ideal generated by $\calE$]
The {\em tree ideal generated by} $\calE$, denoted $\mathcal I_T(\calE)$, is the set of all trees $T\subseteq \omega^{<\omega}$ so that there is a $t \in T$ and a finite set $X \in [\calE]^{<\omega}$ so that $\bigcup T_t \subseteq^* \bigcup X$.

Dually a tree $T \subseteq \omega^{<\omega}$ is in $\mathcal I_T(\calE)^+$ if for each $t \in T$ it's not the case that $\bigcup T_t$ can be almost covered by finitely many functions from $\calE$.
\end{definition}

\begin{remark}
The use of the word ``ideal" here is somewhat misleading as $\mathcal I_T(\calE)$ is not an idea, nor does it generate one. In fact it is not closed under unions. The terminology is meant to draw the analogy with the ideal generated by a MAD family, for which $\mathcal I_T(\calE)$ serves a similar purpose.
\end{remark}

If we have a tree $T \subseteq \omega^{<\omega}$ and a real $g\in \baire$ we say that $g$ {\em densely diagonalizes} $T$ if for each $t \in T$ there is an $s \supsetneq t$ so that $s \in T$ and there is a $k \in {\rm dom}(s) \setminus {\rm dom}(t)$ so that $s(k) = g(k)$. In other words, each node of $T$ sits on a branch which is infinitely often equal to $g$.

The main definition in this paper is the following.
\begin{definition}\label{def.tight}
An eventually different family $\calE$ is {\em tight} if given any sequence of countably many trees $\{T_n \; | \; n < \omega \}$ so that $T_n \in \mathcal I_T(\calE)^+$ for all $n < \omega$ there is a single $g \in \mathcal E$ which densely diagonalizes all the $T_n$'s.
\end{definition}

The point is that this strengthens maximality.
\begin{proposition}\label{if.tight.then.max}
If $\calE$ is tight then it is maximal.
\end{proposition}

\begin{proof}
Suppose $\calE$ is tight but not maximal and let $h \notin \calE$ be eventually different from every element of $\calE$. Consider now the tree $T_h = \{h \hook n \;| \; n < \omega\}$. This tree is in $\mathcal I_T(\calE)^+$ since if we could cover it by finitely many functions from $\calE$ then there would be an $f \in \calE$ so that for infinitely many $n$ $f(n) = (h\hook n+1)(n)$ i.e. $f =^\infty h$. But then by tightness there is a $g \in \calE$ which densely diagonalizes $T_h$. But this just means that $g =^\infty h$, contradiction.
\end{proof}

Later we will see that in $\ZFC$, there are maximal eventually different families that are not tight, see Theorem \ref{definable} below.

\begin{remark}
The intuition behind using trees, as opposed to simply functions is as follows. Often in preservation arguments involving maximal sets in $[\omega]^\omega$ (MAD families, maximal independent families etc) one is given a forcing notion $\P$, a condition $p \in \P$ and a $\P$-name $\dot{X}$ so that $p \forces \dot{X} \in [\omega]^\omega \check{}$ and needs to ``reflect" $\dot{X}$ to the ground model. This is usually done using the {\em outer hull} of $\dot{X}$ with respect to $p$, namely the set $\dot{X}_p := \{m \; | \; p \nVdash \check{m} \notin \dot{X}\}$. The issue in trying to import this idea to a space of functions is that the outer hull of a function (viewed as a subset of $\omega^2$) is no longer necessarily a function. However, if we have a name $\dot{f}$ for an element of $\baire$ and consider instead the set of $\{t \in \omega^{<\omega} \; | \; p \nVdash \check{t} \nsubseteq \dot{f}\}$ then this set forms a subtree of $\omega^{<\omega}$. It is for this reason that we need to work with trees. This idea is central to the argument used to prove Theorem \ref{preservation} below as well as in the applications given in Sections 3 and 6 - 9.
\end{remark}

We need to show that this definition is consistent, i.e. it is consistent that there are tight eventually different families. In fact, much more may be true. 

\begin{theorem}
Assume $\MA(\sigma{\rm -centered})$. Every eventually different family $\calE_0$ of size ${<}2^{\aleph_0}$ is contained in a tight eventually different family. In particular $\CH$ implies that tight eventually different families exist.
\label{construction}
\end{theorem}

The theorem makes use of a forcing notion we introduce now. Let $\calE$ be an eventually different family (not necessarily maximal). Define the forcing notion $\P_{\calE}$ to be the set of all pairs $(s, E)$ so that the following hold.
\begin{enumerate}
\item
$s$ is a finite partial function from $\omega$ to $\omega$
\item
$E \in [\calE]^{<\omega}$
\end{enumerate}

The order on $\P_{\calE}$ is defined as follows. We let $(s_1, E_1) \leq (s_0, E_0)$ if and only if:
\begin{enumerate}
\item
$s_1 \supseteq s_0$ and $E_1 \supseteq E_0$.
\item
If $k \in {\rm dom}(s_1) \setminus {\rm dom}(s_0)$ then $s_1(k) \neq f(k)$ for each $f\in E_0$.
\end{enumerate}
This forcing notion is the same as eventually different forcing $\mathbb E$ from \cite{BarJu95}, but with the second coordinate restricted to $\calE$. It is clear that this forcing notion is $\sigma$-centered since any two conditions $(s, E)$ and $(s, F)$ with the same first coordinate are compatible and in fact strengthened by $(s, E \cup F)$.

We will show that $\P_\calE$ adds a real eventually different from every element of $\calE$ and which densely diagonalizes each $T \in \mathcal I_T(\mathcal E)^+$ in the ground model. This follows from the following lemma.
\begin{lemma}
The following sets are dense in $\P_{\calE}$.
\begin{enumerate}
\item
For each $k < \omega$ the set of conditions $(s, E)$ so that $k \in {\rm dom}(s)$.
\item
For each $T \in \mathcal I_T(\mathcal E)^+$, and $t \in T$ the set of conditions $(s, E)$ so that there is a $t' \in T$ which extends $t$ and a $k \in {\rm dom}(t') \setminus {\rm dom}(t)$ so that $s(k) = t'(k)$.
\item
For each $f \in \calE$ the set of conditions $(s, E)$ so that $f \in E$.
\end{enumerate}
\end{lemma}

Note that condition 1 ensures the generic is a total function on $\omega$, condition 2 ensures the function densely diagonalizes all the trees from $\mathcal I_T(\calE)^+$ in the ground model and condition 3 guarantees that the generic function is eventually different from every element of $\calE$.

\begin{proof}
Let $(s, E)$ be a condition. We will kill three birds with one stone and show that there is a stronger condition in all three dense sets at once. First, by the explanation of $\sigma$-centeredness given above, $(s, E \cup \{f\}) \leq (s, E)$ and is in the last dense set. Now, fix $T \in \mathcal I_T(\calE)^+$ and $t \in T$. By assumption we know that $T_t$ cannot be almost covered by $\bigcup E \cup \{f\}$. It follows that there is a node $t' \supsetneq t$ in $T$ and an $l_t \in {\rm dom}(t') \setminus {\rm dom}(t)$ so that $t'(l_t) \neq f'(l_t)$ for all $f' \in E \cup \{f\}$. Moreover since ${\rm dom}(s)$ is finite we can ensure that $l_t \notin {\rm dom}(s)$. Let $s' = s \cup \{(l_t, t'(l_t) )\}$. Now, fix $k \in \omega$. If $k \notin {\rm dom}(s')$ then let $s'' = s' \cup \{(k, l_k)\}$ where $l_k$ is the minimal element not in the (finite) set $\{f(k) \; | \; f \in E\}$. Finally we get that $(s'', E \cup \{f\} ) \leq (s, E)$, which is what we needed to show.
\end{proof}

Given this we can now prove Theorem \ref{construction}. 

\begin{proof}[Proof of Theorem \ref{construction}]
Assume $\MA(\sigma{\rm -centered})$. Enumerate all $\omega$-sequences of subtrees of $\omega^{<\omega}$ as $\{\vec{T}_\alpha \; | \; \alpha < 2^{\aleph_0}\}$. For each $\alpha < 2^{\aleph_0}$ and $n < \omega$ denote by $T_\alpha(n)$ the $n^{\rm th}$ tree in $\vec{T}_\alpha$. Fix an eventually different family $\calE_0$ of size ${<}2^{\aleph_0}$. We will recursively build a continuous $\subseteq$-increasing sequence of eventually different families $\{\calE_\alpha \; | \; \alpha < 2^{\aleph_0}\}$ so that for each $\alpha < 2^{\aleph_0}$ $\calE_\alpha$ has cardinality ${<}2^{\aleph_0}$ and if $\vec{T}_\alpha \in \mathcal I_T(\calE_\alpha)^+$ then there is a real in $\calE_{\alpha+1}$ densely diagonalizing $T_\alpha(n)$ for each $n < \omega$. If such a sequence can be constructed then clearly $\calE = \bigcup_{\alpha < 2^{\aleph_0}} \calE_\alpha$ will be the desired tight eventually different family. 

We already have $\calE_0$. Now suppose that we have constructed $\calE_\alpha$. If $\vec{T}_\alpha \nsubseteq \mathcal I_T(\calE_\alpha)^+$ then let $g_\alpha$ be any function eventually different from every element of $\calE_\alpha$. That such an element exists is guaranteed by $\MA(\sigma{\rm -centered})$ (using eventually different forcing). If $\vec{T}_\alpha \subseteq \mathcal I_T(\calE_\alpha)^+$ then, by applying $\MA (\sigma {\rm -centered})$ to $\P_{\calE_{\alpha}}$ and noting that we only need to meet $|\calE_\alpha| + \aleph_0$ dense sets, find a $g_\alpha$ eventually different from each element of $\calE_\alpha$ and densely diagonalizing every $T_\alpha(n)$ for $n < \omega$. In either case let $\calE_{\alpha+1} = \calE_\alpha \cup \{g_\alpha\}$.  Either way it's clear that we have fulfilled the requisite conditions.
\end{proof} 

Of course a natural question is whether tight eventually different families exist in $\ZFC$. We do not know the answer to this. However the analogous question for tight MAD families, namely if there is a tight MAD family in $\ZFC$, remains one of the most stubbornly open problems in this area so it's reasonable to expect that finding a $\ZFC$ example of a tight eventually different family, or showing one does not exist may be difficult as well.

\begin{question}
Is it consistent with $\ZFC$ that there are no tight maximal eventually different families?
\end{question}

\section{Cohen Indestructibility and Definability}
Let us investigate the analogy between tight eventually different families and tight MAD families further. As mentioned in the introduction, in \cite[Corollary 3.2]{orderingMAD} it is shown that tight MAD families are Cohen indestructible. Here, recall that if $\P$ is a forcing notion and $\calA$ is MAD then it is said to $\P$-{\em indestructible} if $\forces_\P$ ``$\check{\calA}$ is MAD".  Here, we show that the same holds true of tight eventually different families. We then use this to show that there are no analytic tight eventually different families and, under large cardinals, no definable such families in a strong sense.

\begin{theorem}
Suppose $\kappa$ is a cardinal and $\mathbb C_\kappa = add(\omega, \kappa)$ is the forcing to add $\kappa$-many Cohen reals. If $\calE$ is tight then $\forces_{\mathbb C_\kappa}$ ``$\check{\calE}$ is tight".
\label{cohenindestructible}
\end{theorem}

\begin{proof}
Since every new real in the Cohen extension is added by a single Cohen real, it suffices to prove the theorem in the case $\kappa = 1$. Denote $\mathbb C_1$ by $\mathbb C$. Let $\calE$ be tight and let $\{\dot{T}_n \; | \; n < \omega\}$ be a countable set of $\mathbb C$-names for subtrees of $\omega^{<\omega}$. We will show that for each $p \in \mathbb C$ either there is a $q \leq p$, an $n < \omega$ so that $q \forces \dot{T}_n \in \mathcal I_T(\calE)\check{}$ or there is a $g \in \calE$ so that $p \forces$ ``For each $n < \omega$, $\check{g}$ densely diagonalizes $\dot{T}_n$". Clearly the theorem will follow from this.

Fix $p \in \mathbb C$ and enumerate the set of conditions below $p$ as $\{p_j \; | \; j < \omega\}$. For each $n, j < \omega$ let $T_{n, j} = \{t \in \omega^{<\omega} \; |\; \exists q \leq p_j \, q \forces \check{t} \in \dot{T}_n \}$. For each $n, j < \omega$ the set $T_{n, j}$ is a tree since if some $t \in T_{n, j}$ as witnessed by some $q$ then the same $q$ witnesses that $s \in T_{n, j}$ for each $s \subseteq t$. There are two cases.

\noindent \underline{Case 1}: There are $n, j< \omega$ and $t\in T_{n, j}$ so that $(T_{n, j})_t$ is almost covered by finitely many functions from $\calE$, i.e. $T_{n, j} \in \mathcal I_T(\calE)$. Fix $q \leq p_j$ so that $q$ witnesses that $t \in T_{n, j}$ and suppose that $f_0, ..., f_{n-1} \in \calE$ are such that $\bigcup (T_{n, j})_t \subseteq^* \bigcup\{f_0, ..., f_{n-1}\}$. Unwinding the definition of $T_{n, j}$ we get that in this case $q$ forces $\emptyset \neq (\dot{T}_n)_{\check{t}} \subseteq (T_{n, j})_t$ since if $t'$ is compatible with $t$ and some $r \leq q$ forces $\check{t}' \in \dot{T}_n$ then $r$ witnesses that $t' \in (T_{n, j})_t$ and since $q$ forces that $t \in \dot{T}_n$ then $q$ forces in particular that $(\dot{T}_n)_t$ is non empty. It follows that $q$ forces that $\bigcup (\dot{T}_n)_t \subseteq^* \bigcup\{\check{f}_0, ..., \check{f}_{n-1}\}$. In other words, $q$ must force that $\dot{T}_n$ is in $\mathcal I_T(\calE)$. 

\noindent \underline{Case 2}: Case 1 fails. In other words this means that each $T_{n, j}$ is not in $\mathcal I_T(\calE)$. By the assumption that $\calE$ is tight, this means that there is a single $g \in \calE$ which densely diagonalizes all of the $T_{n, j}$'s. We claim that $p \forces$ ``For all $n < \omega$ $g$ densely diagonalizes $\dot{T}_n$". To see this, fix $n < \omega$ and let $q = p_m < p$ for some $m$. Suppose $\dot{t}$ is a name for an element of $\dot{T}_n$. By strengthening $q$ if necessary we can assume that $\dot{t}$ is decided by $q$ to be some $t \in T_{n, m}$. Moreover, since $g$ densely diagonalizes $T_{n, m}$ there is an $s \supseteq t$ which is in $T_{n, m}$ as witnessed by some $r\leq q$ and an $l \in {\rm dom}(s) \setminus {\rm dom}(t)$ and $s(l) = g(l)$. It follows that $r$ forces there to be a node of $\dot{T}_n$ strengthening $\dot{t}$ and agreeing with $g$ above the domain of $\dot{t}$. Since $q$ was arbitrary it follows that $p$ actually forces this for each $\dot{t}$ and $n$ and hence $p$ forces $g$ to densely diagonalize each $\dot{T}_n$ as needed.
\end{proof}

It follows from this result and Theorem \ref{pi11medf} below that there is a coanalytic Cohen indestructible maximal eventually different family. This was known, \cite{indestmcg}, but only by a more complicated construction. We also have the following result.
\begin{corollary}
In the Cohen model there is a tight eventually different family of size $\aleph_1$.
\end{corollary}

We can use this theorem to show that not every maximal eventually different family is tight. In fact something stronger is true.
\begin{theorem}
If $\calE$ is an analytic maximal eventually different family and $\P$ is any forcing notion adding a real then $\forces_\P$ ``$\check{\calE}$ is no longer maximal". In particular, tight eventually different families are never analytic.
\label{definable}
\end{theorem}

Since there are Borel maximal eventually different families, see \cite{BorelMedf}, these ones are not tight. 

\begin{proof}[Proof of Theorem \ref{definable}]
Let $\mathcal A$ be an analytic, maximal eventually different family. Note that $\mathcal A$ is uncountable and therefore, by the perfect set property for analytic sets, it contains a perfect set $P \subseteq \mathcal A$. Moreover, since every perfect set contains a copy of Cantor space, we can find a continuous injection $f:\cantor \to \mathcal A$. Fix such an $f$. 

Observe that ``$\mathcal A$ is an eventually different family" is $\Pi^1_1$: namely we have $\forall x, y [ x \notin \mathcal A \lor y \notin \mathcal A \lor x \neq^* y]$. Since $\mathcal A$ is analytic, and hence its complement is co-analytic the observation follows. In particular, the interpretation of $\mathcal A$ remains an eventually different family in any forcing extension. Now let $\P$ be a forcing notion adding a real, $x$. Working in $V^\P$ we consider the real $f(x) \in \mathcal A^{V^\P}$. This real is eventually different from every element of $\mathcal A \cap V$ and hence $\mathcal A \cap V$ is no longer maximal. 
\end{proof}

Observe the theorem only used the perfect set property for $\mathcal A$ alongside some generic absoluteness. As a result the same proof gives the following.

\begin{theorem}
\begin{enumerate}
\item
If $E$ is a co-analytic maximal eventually different family which is indestructible with respect to some forcing adding a real then $E$ does not contain a perfect set.

\item
In the presence of sufficient large cardinals there are no $\P$-indestructible maximal eventually different families in $L(\mathbb R)$ for any forcing notion $\P$ adding a real.
\end{enumerate}
\end{theorem}

Note also that all of these results apply equally in the case of cofinitary groups and maximal eventually different sets of permutations. In particular, there is no analytic, maximal cofinitary group which is indestructible for any forcing notion adding a real.

\section{Strong Preservation of Tightness}
The purpose of this section is to prove a preservation theorem for tight, eventually different families, akin to \cite[Corollary 32]{restrictedMADfamilies} which showed the same for tight MAD families. The preservation theorem we prove concerns the notion of {\em strong preservation}. 

\begin{definition}
Let $\P$ be a proper forcing notion and $\calE$ a tight eventually different family. We say that $\P$ {\em strongly preserves the tightness of} $\calE$ if for every sufficiently large $\theta$, every condition $p$ and every $M \prec H_\theta$ countable with $p, \P, \calE \in M$, if $g \in \calE$ strongly diagonalizes every element of $M \cap \mathcal I_T(\calE)^+$ then there is an $(M, \P)$-generic $q \leq p$ so that $q$ forces that $g$ densely diagonalizes every element of $M[\dot{G}] \cap \mathcal I_T(\calE)^+$. Such a $q$ is called an $(M, \P, \calE, g)$-generic condition. 
\label{strongpres}
\end{definition}

Clearly if $\P$ strongly preserves the tightness of $\calE$ then in particular $\forces_\P$`` $\check{\calE}$ is tight". The point of this section is to prove the following.

\begin{theorem}
Suppose that $\calE$ is a tight eventually different family. If $\langle \P_\alpha, \dot{\mathbb Q}_\alpha \; | \; \alpha< \gamma\rangle$ is a countable support iteration of proper forcing notions so that for all $\alpha$ we have $\forces_\alpha$ ``$\dot{\mathbb Q}_\alpha$ strongly preserves the tightness of $\check{\calE}$" then $\P_\gamma$ strongly preserves the tightness of $\calE$.
\label{preservation}
\end{theorem}

This theorem follows immediately from the following two lemmas. From now on fix a tight eventually different family $\calE$.

\begin{lemma}
Suppose $\P$ strongly preserves the tightness of $\calE$ and $\dot{\mathbb Q}$ is a $\P$-name for a poset which strongly preserves the tightness of $\calE$. Then $\P * \dot{\mathbb Q}$ strongly preserves the tightness of $\calE$. Moreover if $p$ is $(M, \P, \calE, g)$-generic and forces $\dot{q}$ to be $(M[\dot{G}], \dot{\Q}, \calE, g)$-generic then $(p, \dot{q})$ is $(M, \P * \dot{\mathbb Q}, \calE, g)$-generic.
\label{successorstep}
\end{lemma}

\begin{proof}
Suppose $p$ is $(M, \P, \calE, g)$-generic and forces $\dot{q}$ to be $(M[\dot{G}], \dot{\Q}, \calE, g)$-generic. Then obviously $(p, \dot{q})$ is $(M, \P * \dot{\mathbb Q})$-generic. Moreover, by definition it forces that for every $\P$-name for a $\dot{\mathbb Q}$-name for an element of $\mathcal I_T(\calE)^+$ in $M$ $g$ densely diagonalizes it. But therefore $(p, \dot{q})$ forces that for every $\P * \dot{\mathbb Q}$-name for an element of $\mathcal I_T(\calE)^+$ in $M$ is densely diagonalized by $g$ as needed.
\end{proof}

\begin{lemma}
Let $\langle \P_\alpha, \dot{\mathbb Q}_\alpha \; | \; \alpha< \gamma\rangle$ be a countable support iteration of proper forcing notions so that for all $\alpha$ we have $\forces_\alpha$ ``$\dot{\mathbb Q}_\alpha$ strongly preserves the tightness of $\check{\calE}$", let $\theta$ be sufficiently large, $M \prec H_\theta$ countable containing $\P_\gamma, \gamma, \calE$. For each $\alpha \in M \cap \gamma$ and every $(M, \P_\alpha, \calE, g)$-generic condition $p \in \P_\alpha$ the following holds:

If $\dot{q}$ is a $\P_\alpha$-name $p\forces_\alpha \dot{q} \in \P_\gamma \cap M$ and $p\forces_\alpha \dot{q}\hook \alpha \in \dot{G}_\alpha$ then there is an $(M, \P_\gamma, \calE, g)$-generic condition $\bar{p} \in \P_\gamma$ so that $\bar{p} \hook \alpha = p$ and $\bar{p} \forces_\gamma \dot{q} \in \dot{G}$. 
\label{limitstep}

\end{lemma}

Before proving this lemma, let us fix ahead of time a convention regarding the enumeration of trees $T \subseteq \omega^{<\omega}$. First fix a computable enumeration of $\omega^{<\omega}$ so that shorter sequences appear first. Now, given a tree $T \subseteq \omega^{<\omega}$ we can push that enumeration forward onto $T$ in the sense that the $i^{\rm th}$ node of $T$ is the $i^{\rm th}$ node in $T$ relative to the computable enumeration of $\omega^{<\omega}$ we fixed. In other words if our computable enumeration of $\omega^{<\omega}$ is, say $\{t_k \; | \; k < \omega\}$ then the $0^{\rm th}$ node of $T$ is $t_i$ so that $i$ is least with $t_i \in T$ and so on. In this way we can refer unambiguously to an enumeration of the nodes of a tree in a forcing extension. This convention will also be used in the remaining sections of the paper as well.


\begin{proof}
The proof is by induction on $\gamma$. The case where $\gamma$ is a successor ordinal follows from Lemma \ref{successorstep} so we focus on the limit case. Fix $\alpha, p, M$ etc as in the statement of the lemma and let $\{\alpha_n \; | \; n  <\omega\}$ be a strictly increasing sequence in $M\cap \gamma$ with supremum $\gamma$ so that $\alpha_0 = \alpha$. Fix a bijection $\phi:\omega \to \omega^2$ with coordinate functions $\phi_0$ and $\phi_1$. Enumerate the dense open subsets of $\mathbb P_\gamma$ in $M$ as $\{D_n\; | \; n < \omega\}$. Enumerate the $\P_\gamma$-names for $\mathcal I_T(\calE)^+$ trees in $M$ as $\{\dot{T}_n \; | \; n < \omega\}$. We will recursively define sequences $\{p_n \; | \; n < \omega\}$, $\{\dot{q}_n \; | \; n < \omega\}$, $\{\dot{t}_n \; | \; n < \omega \}$ and $\{\dot{k}_n \; | \; n < \omega \}$ as follows.
\begin{enumerate}
\item
$p_0 = p$ and $\dot{q}_0 = \dot{q}$
\item
$p_n$ is an $(M, \P_{\alpha_n}, \calE, g)$-generic condition
\item
$p_{n+1} \hook \alpha_n = p_n$
\item
$\dot{q}_n$ is a $\P_{\alpha_n}$ name so that $p_n \forces_{\alpha_n} \dot{q}_n \in \P_\gamma \cap M$ and $p_n \forces_{\alpha_n} \dot{q}_n \hook \alpha_n \in \dot{G}_{\alpha_n}$
\item
$p_{n+1} \forces_{\alpha_{n+1}} \dot{q}_{n+1} \leq \dot{q}_n$ and $p_{n+1} \forces_{\alpha_{n+1}} \dot{q}_{n+1} \in D_n$
\item
$\dot{t}_n$ is a $\P_{\gamma}$ name for a node in $\dot{T}_{\phi_0(n)}$ strictly above the $\phi_1(n)^{\rm th}$ node in $\dot{T}_{\phi_0(n)}$
\item
$\dot{k}_n$ is a name for an element of $\omega$ forced to be in the domain of $\dot{t}_n$ above the domain of the $\phi_1(n)^{\rm th}$ node in $\dot{T}_{\phi_0(n)}$ and $p_n \forces_{\alpha_n}$``$\dot{q}_n \forces_\gamma \check{g}(\dot{k}_n) = \dot{t}_n(\dot{k}_n)$"
\end{enumerate}

Assuming that such a sequence can be constructed we let $\bar{p} = \bigcup_{n < \omega} p_n$. Clearly this condition is as required. Therefore it remains to see that we can construct such a quadruple of sequences. This is done by induction. The base case is given. Suppose that we have constructed for some $n < \omega$ sequences $\{p_m \; | \; m < n+1\}$, $\{\dot{q}_m \; | \; m < n+1\}$, $\{\dot{t}_m \; | \; m < n+1 \}$ and $\{\dot{k}_m \; | \; m < n+1 \}$ satisfying the above conditions and we construct $p_{n+1}$, $\dot{q}_{n+1}$, $\dot{t}_{n+1}$ and $\dot{k}_{n+1}$. 

Let $p_n \in G_{\alpha_n}$ be generic and work in $V[G_{\alpha_n}]$. Let $q_n \in \P_\gamma \cap M$ be the evaluation of the name $\dot{q}_n$ by $G_{\alpha_n}$ so that $q_n \hook \alpha_n \in G_{\alpha_n}$. Since $p_n$ is $(M, \P_{\alpha_n})$ generic there is an $r \in D_n \cap M$ so that $r\leq q_n$ and $r \hook \alpha_n \in G_{\alpha_n}$. Without loss we may assume that $r$ decides that the $\phi_1(n)^{\rm th}$ node in $\dot{T}_{\phi_0(n)}$ to be some $s_n \in \omega^{<\omega}$. Now let $W$ be the set of all $t \in {\omega}^{<\omega}$ so that there is an $\bar{r} \leq r$ so that $\bar{r} \in \P_\gamma$ and the following hold.
\begin{enumerate}
\item
$\bar{r} \hook \alpha_n \in G_{\alpha_n}$
\item
$\bar{r} \forces_\gamma$ ``$\check{t} \in \dot{T}_{\phi_0(n)}$ and $\check{t} \supsetneq \check{s}_n$"
\end{enumerate}

Observe that, by appending $s_n$ and its predecessors to $W$ we have a tree with stem (including) $s_n$. This is because if some $\bar{r}$ forces $t \in \dot{T}_{\phi_0(n)}$ to be strictly above $s_n$ and $s_n \supsetneq t' \supsetneq t$ then trivially the same $\bar{r}$ forces the same of $t'$ so $W$ is closed downwards under sequences extending $s_n$. Moreover $W \in M[G_\alpha]$ by construction (since $\dot{T}_{\phi_0(n)}$ and $r$ are) and, crucially $W \in \mathcal I_T(\calE)^+$. To see this last point observe that if $r \in G$ is $\P_\gamma$ generic then, in $V[G]$ we must have that the evaluation of $\dot{T}_{\phi_0(n)}$ is contained in $W$ so if $W$ could be covered by finitely many functions from $\calE$ then so could $\dot{T}_{\phi_0(n)}$ contradicting the fact that it was forced to be a tree in $\mathcal I_T(\calE)^+$. 

We can therefore apply assumption (2) for $p_n$ and conclude that $g$ densely diagonalizes $W$. It follows that there is a node $t\in W$ so that $t \supsetneq s_n$ and there is a $k \in {\rm dom}(t) \setminus {\rm dom}(s_n)$ so that $g(k) = t(k)$. Let $\dot{q}_{n+1}$ be a name for the $\bar{r}$ witnessing that $t \in W$ and let $\dot{t}_n$ and $\dot{k}_n$ be names for $t$ and $k$ back in $V$. Finally, by our inductive hypothesis we can find a $(M, \P_{\alpha_{n+1}}, \calE, g)$-generic condition $p_{n+1} \leq p_n$ as needed.
\end{proof}

Sacks forcing strongly preserves the tightness of any tight eventually different family. This gives an alternative proof of the fact that $\mfa_e = \aleph_1$ in the Sacks model. The argument for this is very similar to the analogous one we give for Miller forcing in Section 5 so we leave the details of this to the reader. In Sections 6-8 we also show that partition forcing, infinitely often equal forcing and Shelah's $\mathbb Q_\mathcal I$ forcing all strongly preserves the tightness of any given tight eventually different family. Before giving these arguments we turn to a slightly different type of eventually different family.
 
\section{Tight Eventually Different Permutations}
The foregoing discussion of tight eventually different families works mutatis mutandis for {\em eventually different sets of permutations} of $\omega$. Recall that $\mfa_p$ is the least size of a maximal eventually different family of permutations. It remains an intriguing open question if $\ZFC$ proves $\mfa_e = \mfa_p$. Let us give the ``permutation" version of the above definitions and state the analogous results. Since the proofs are almost verbatim the same we simply indicate the necessary changes and leave the details to the reader.

Call a tree $T \subseteq \omega^{<\omega}$ {\em injective} if each $t \in T$ is injective as a finite function. Let $\calP$ be a family of eventually different permutations.
\begin{definition}
The {\em tree ideal generated by} $\calP$, denoted $\mathcal I_T(\calP)$, is the set of all {\bf injective} trees $T\subseteq \omega^{<\omega}$ so that there is a $t \in T$ and a finite set $X \in [\calE]^{<\omega}$ so that $\bigcup T_t \subseteq^* \bigcup X$.

Dually an injective tree $T \subseteq \omega^{<\omega}$ is in $\mathcal I_T(\calP)^+$ if for each $t \in T$ it's not the case that $\bigcup T_t$ can be almost covered by finitely many functions from $\calP$.
\end{definition}

The definition of tightness for sets of permutations is now identical to that of functions.
\begin{definition}
An eventually different set of permutations $\calP$ is {\em tight} if given any sequence of countably many injective trees $\{T_n \; | \; n < \omega \}$ so that $T_n \in \mathcal I_T(\calP)^+$ for all $n < \omega$ there is a single $g \in \mathcal P$ which densely diagonalizes all the $T_n$'s.
\end{definition}

The same proof as before shows that:
\begin{proposition}\label{if.tight.then.maximal}
If $\calP$ is tight then it is maximal.
\end{proposition}

We also have the analogue of Theorem \ref{construction}
\begin{theorem}\label{MA.tight}
Assume $\MA(\sigma{\rm -centered})$. Every eventually different family $\calP_0$ of permutations of size ${<}2^{\aleph_0}$ is contained in a tight eventually different set of permutations. In particular $\CH$ implies that tight eventually different sets of permutations exist.
\end{theorem}

\begin{proof}
The proof is almost the same as that of Theorem \ref{construction}. We just indicate what the right poset is and the necessary modifications.  Let $\calP$ be an eventually different set of permutations (not necessarily maximal). Define the forcing notion $\Q_{\calP}$ to be the set of all pairs $(s, E)$ so that the following hold.
\begin{enumerate}
\item
$s$ is an injective finite partial function from $\omega$ to $\omega$
\item
$E \in [\calP]^{<\omega}$
\end{enumerate}

The order on $\Q_{\calP}$ is defined exactly the same as for $\P_{\calE}$. Mimicking the proof of Theorem \ref{construction}, it's enough to show that the following sets are dense.
\begin{enumerate}
\item
For each $k < \omega$ the set of conditions $(s, E)$ so that $k \in {\rm dom}(s)$.
\item
For each $k < \omega$ the set of conditions $(s, E)$ so that $k \in {\rm range}(s)$.
\item
For each $T \in \mathcal I_T(\mathcal P)^+$, and $t \in T$ the set of conditions $(s, E)$ so that there is an $t' \in T$ which extends $t$ and a $k \in {\rm dom}(t') \setminus {\rm dom}(t)$ so that $s(k) = t'(k)$.
\item
For each $f \in \calE$ the set of conditions $(s, E)$ so that $f \in E$.
\end{enumerate}
This is shown exactly as for eventually different families, the one caveat being that we need to be able ensure that when we strengthen some $(s, E)$ we can guarantee that if $T \in \mathcal I_T(\calP)^+$ then we can find an $t' \supsetneq t$ so that $t' \in T$ and a $k \in {\rm dom}(t') \setminus {\rm dom} (t) \cap {\rm dom}(s)$ so that we can add $(k, t'(k))$ to $s$ without wrecking injectivity. This is where we use that fact that $T$ is injective. Namely, we can find a level in $T$ above which no element of any sequence is in the range of $s$ (since $s$ is finite) and therefore we can find the needed $k$. Everything else in the proof is exactly as in the case of eventually different families of functions.
\end{proof}

We also have the same indestructibility results.
\begin{theorem}
Suppose $\kappa$ is a cardinal and $\mathbb C_\kappa = add(\omega, \kappa)$ is the forcing to add $\kappa$-many Cohen reals. If $\calP$ is a tight eventually different set of permutations then $\forces_{\mathbb C_\kappa}$ ``$\check{\calP}$ is tight".
\label{cohenindestructible2}
\end{theorem}

Similarly, there is an analogue of Theorem \ref{preservation} for tight sets of permutations. We give the definition and theorem below.
\begin{definition}
Let $\P$ be a proper forcing notion and $\calP$ a tight eventually different set of permutations. We say that $\P$ {\em strongly preserves the tightness of} $\calP$ if for every sufficiently large $\theta$, every condition $p$ and every $M \prec H_\theta$ countable with $p, \P, \calP \in M$, if $g \in \calP$ densely diagonalizes every element of $M \cap \mathcal I_T(\calP)^+$ then there is an $(M, \P)$-generic $q \leq p$ so that $q$ forces that $g$ densely diagonalizes every element of $M[\dot{G}] \cap \mathcal I_T(\calP)^+$. Such a $q$ is called an $(M, \P, \calP, g)$-generic condition. 
\end{definition}

\begin{theorem}
Suppose that $\calP$ is a tight eventually different family. If $\langle \P_\alpha, \dot{\mathbb Q}_\alpha \; | \; \alpha< \gamma\rangle$ is a countable support iteration of proper forcing notions so that for all $\alpha$ we have $\forces_\alpha$ ``$\dot{\mathbb Q}_\alpha$ strongly preserves the tightness of $\check{\calP}$" then $\P_\gamma$ strongly preserves the tightness of $\calP$.
\label{preserve2}
\end{theorem}

The proof of Theorems \ref{cohenindestructible2} and \ref{preserve2} are identical to those of Theorems \ref{cohenindestructible} and \ref{preservation} so we leave the details to the reader. The one point to note is that the $T_{n, j}$'s of Theorem \ref{cohenindestructible} and the $W$ found in the inductive proof of Lemma \ref{limitstep} are both injective since $\dot{T}_n$ (respectively $\dot{T}_{\phi_0(n)}$) is forced to be injective hence if some $r \forces \check{t} \in \dot{T}_n$ (respectively $\dot{T}_{\phi_0(n)}$) then it must be the case that $t$ is injective.

As in this section every further one the proofs involving tight eventually different families and tight eventually different permutations are almost identical and, largely we prove the case of tight eventually different families in detail and leave the the case of tight eventually different permutations to the reader. This situation adds fuel to the idea that perhaps $\ZFC$ proves that $\mfa_e = \mfa_p$ and in any case shows how difficult it could be to separate these two invariants even if this ends up being possible. 

\section{Miller Forcing and the Consistency of $\mfa_e = \mfa_p < \mfd = \mfa_T$}
In this section we show that Miller forcing strongly preserves the tightness of any tight eventually different family of functions and any tight eventually different set of permutations. As a result we obtain the consistency of $\mfa_e = \mfa_p < \mfd = \mfa_T$. Recall that {\em Miller forcing}, also called rational perfect set forcing and denoted $\mathbb{PT}$ consists of all trees $T \subseteq \omega^{<\omega}$ so that for every node $t \in T$ there is an $s \in T$ extending $t$ with infinitely many immediate successors. The order is inclusion. As is often done we work with the dense subset of trees in which every node has either one or infinitely many immediate successors. It's well known that this forcing is proper and, when iterated $\omega_2$-many times with countable support produces a model of $non(\Me) = cov(\Me) = \aleph_1 < \mfd = \aleph_2$ (and $\mfa_T = \aleph_2$ as well). More information about Miller forcing can be found in \cite{BarJu95}, see in particular Definition 7.3.43. 

\begin{theorem}
Let $\calE$ be a tight eventually different family. Miller forcing $\mathbb{PT}$ strongly preserves the tightness of $\calE$.
\label{miller}
\end{theorem}

Before proving this theorem we recall some basic terminology. Recall that if $p \in \mathbb{PT}$ is a Miller tree and $n < \omega$ then a node $t \in p$ is an ${n}^{\rm th}${\rm -splitting node} if it has infinitely many immediate successors and it has the $n - 1$  predecessors with this property. Denote by ${\rm Split}_n(p)$ the set of $n$-splitting nodes. We say that for two Miller trees $p, q \in\mathbb{PT}$ that $q \leq_n p$ if $q \leq p$ and ${\rm Split}_n(p) ={\rm Split}_n(q)$. 

\begin{proof}
 Let $p\in \mathbb{PT}$ be a condition, let $M\prec H_\theta$ countable with $\theta$ sufficiently large, $p, \mathbb{PT}, \calE \in M$. Let $g \in \calE$ densely diagonalize every $T \in M \cap \mathcal I_T(\calE)^+$. Let $\{T_n \; | \; n < \omega\}$ be an enumeration of all $\mathbb{PT}$ names in $M$ for trees in $\mathcal I_T(\calE)^+$. Let $\{D_n \; | \; n < \omega\}$ enumerate all dense open subsets of $\mathbb{PT}$ in $M$. Let $\phi:\omega \to \omega^2$ be a bijection with coordinate maps $\phi_0$ and $\phi_1$. Inductively we will construct sequences $\{p_n \; | \; n < \omega\}$, $\{\dot{t}_n \; | \; n < \omega\}$, $\{\dot{k}_n \; | \; n < \omega\}$ so that the following hold.
\begin{enumerate}
\item
$p_0 = p$
\item
$\{p_n \; |\; n < \omega\} \subseteq M$ and for all $n <\omega$, $p_{n+1} \leq_{n+1} p_n$
\item
For each $n< \omega$, $\dot{t}_n$ is a $\mathbb{PT}$ name in $M$ for a node in $\dot{T}_{\phi_0(n)}$ extending the $\phi_1(n)^{\rm th}$-node in $\dot{T}_{\phi_0(n)}$
\item
For each $n < \omega$, $\dot{k}_n$ is a name for an element of $\omega$ in $M$
\item
For each ${n+1}^{\rm st}$-splitting node $t$ of $p_{n+1}$ we have that $(p_{n+1})_t \in D_n$, and forces for some $s_n \in \omega^{<\omega}$ that the $\phi_1(n)^{\rm th}$-node in $\dot{T}_{\phi_0(n)}$ is $\check{s}_n$ and $(p_{n+1})_t \forces \dot{k}_n \in {\rm dom}(\dot{t}_n) \setminus {\rm dom}(\check{s}_n) \land \dot{t}_n(\dot{k}_n) = \check{g} (k)$.
\end{enumerate}

Suppose first that we can construct such a sequence. Let $q = \bigcap_{n < \omega} p_n$. It follows almost immediately that $q$ is $(M, \mathbb{PT}, \calE, g)$-generic as needed.

Therefore it remains to construct the requisite sequence. This is done by induction. Let $p_0 = p$. Now assume that $\{p_j \; | \; j < n+1\}$, $\{\dot{t}_j \; | \; j < n\}$ and $\{\dot{k}_j \; | \; j < n\}$ have been defined. We will define $p_{n+1}$, $\dot{t}_n$ and $\dot{k}_n$. For each ${n+1}^{\rm st}$-splitting node $t$ of $p_n$, let $q_t \leq (p_n)_t$ be an element of $M \cap D_n$ which forces for some $s_n \in \omega^{<\omega}$ that the $\phi_1(n)^{\rm th}$-node in $\dot{T}_{\phi_0(n)}$ is $\check{s}_n$. Let $W_t = \{u \supsetneq s_n \; | \; q_t \nVdash u \notin \dot{T}_{\phi_0(n)}\}$. As in the proof of Lemma \ref{preservation} $W_t$ alongside $s_n$ and its predecessors is a tree in $\mathcal I_T(\calE)^+ \cap M$ hence it is densely diagonalized by $g$. Therefore we can find a $u_t \in W_t$ and a $k_t \in {\rm dom}(u_t) \setminus {\rm dom}(s_n)$ so that $u_t(k_t) = g(k_t)$. Let $r_t \leq q_t$ force that $\check{u_t} \in \dot{T}_{\phi_0(n)}$. Now let $p_{n+1} = \bigcup_{t \in {\rm Split}_{n+1}(p_n)} r_t$. Let $\dot{t}_n = \{(\check{u}_t, r_t) \; | \; t \in {\rm Split}_{n+1}(p_n)\}$ and $\dot{k}_n =\{(\check{k}_t, r_t) \; | \; t \in {\rm Split}_{n+1}(p_n)\}$. Clearly these suffice.

\end{proof}

Essentially the same proof gives the analogous result for tight eventually different families of permutations. We state the theorem below and leave the (primarily cosmetic) modifications of the above proof to the reader.

\begin{theorem}
Let $\calP$ be a tight eventually different family of permutations. Miller forcing strongly preserves the tightness of $\calP$.
\end{theorem}

As a result we get the following.
\begin{theorem}
In the iterated Miller model the inequality $\mfa_e = \mfa_p < \mfd = \mfa_T$ holds.
\label{miller2}
\end{theorem}

Note that this inequality also holds in the Cohen model, however unlike the Cohen model, this model $cov(\mathcal M) = \aleph_1$ so we can in fact also gain freedom over $cov(\mathcal M)$. Also, as mentioned in the introduction, $\mfa_p = \aleph_1$ in the Miller model was originally shown by Kastermans and Zhang using parametrized diamonds in \cite{KastermansZhang06}. 

\section{Partition Forcing and the Consistency of $\mfa_e = \mfa_p = \mfd < \mfa_T$}

In this section we show that Miller's partition forcing strongly preserves the tightness of any tight eventually different family of functions and any tight eventually different family of permutations. As a result we obtain the consistency of $\mfa_e = \mfa_p = \mfd < \mfa_T$. The proof of the preservation result mirrors the analogous one for tight MAD families given in \cite[Proposition 38]{restrictedMADfamilies}. We recall some terminology used there.

Recall from the introduction that if $\mathcal K$ is an uncountable partition of $\cantor$ into closed sets then $\P(\mathcal K)$, the partition forcing, is the set of perfect trees $p$ so that for all $C \in \mathcal K$ we have $[p] \cap C$ is nowhere dense in $[p]$. This forcing was introduced in \cite{MillerCOV} to increase $\mfa_T$. In order to ensure that a perfect tree is in $\P(\mathcal K)$ we need the definition of a {\em nice} set of reals.

\begin{definition}[Definition 36 of \cite{restrictedMADfamilies}]
Fix a partition of $\cantor$ into closed sets $\mathcal K = \{C_\alpha \; | \; \alpha < \omega_1\}$. We say that $X = \{x_s \; | \; s \in \omega^{<\omega}\} \subseteq \cantor$ is {\em nice} (for $\mathcal K$) if the following conditions holds:
\begin{enumerate}
\item
For every $s \in\omega^{<\omega}$ the sequence $\langle x_{s^\frown n}\rangle_{n < \omega}$ converges to $x_s$ and $\Delta(x_s, x_{s^\frown n}) < \Delta (x_s, x_{s^\frown n+1})$ for all $n < \omega$\footnote{Recall that if $x\neq y \in \cantor$ then $\Delta(x, y)$ is the least $k \in \omega$ so that $x(k) \neq y(k)$.}.
\item
For all $s, t, z \in \omega^{<\omega}$ if $s\subsetneq t \subsetneq z$ then $\Delta(x_s, x_t) < \Delta (x_t, x_z)$.
\item
For every $s \in \omega^{<\omega}$ let $\alpha_s < \omega_1$ so that $x_s \in C_{\alpha_s}$. If $s\subsetneq t$ then $\alpha_s \neq \alpha_t$.
\end{enumerate}
\end{definition}

The point is the following Lemma.
\begin{lemma}[Lemma 37 of \cite{restrictedMADfamilies}]
Let $p$ be a Sacks tree and let $\mathcal K$ be an uncountable partition of $\cantor$ into closed sets. If there is an $X = \{x_s\; | \; s \in \omega^{<\omega}\}$ which is nice for $\mathcal K$ and dense in $[p]$ then $p \in \P(\mathcal K)$.
\end{lemma}

Armed with these facts we can prove the main theorem of this section. Our proof is extremely similar to the analogous proof for tight MAD families, \cite[Proposition 38]{restrictedMADfamilies}. The modifications follow those of Theorem \ref{miller}.

\begin{theorem}
Let $\mathcal K \subseteq P(\cantor)$ be an uncountable partition of $\cantor$ into closed sets and let $\calE$ be a tight eventually different family. The forcing notion $\P(\mathcal K)$ strongly preserves the tightness of $\calE$.
\label{partitionpreservation1}
\end{theorem}

\begin{proof}
Let $p\in \P(\mathcal K)$ be a condition, let $M\prec H_\theta$ countable with $\theta$ sufficiently large, $p, \P(\mathcal K), \mathcal K, \calE \in M$. Let $g \in \calE$ densely diagonalize every $T \in M \cap \mathcal I_T(\calE)^+$. Let $\{T_n \; | \; n < \omega\}$ be an enumeration of all $\P(\mathcal K)$ names in $M$ for trees in $\mathcal I_T(\calE)^+$. Let $\{D_n \; | \; n < \omega\}$ enumerate all dense open subsets of $\P(\mathcal K)$ in $M$. Let $\varphi:\omega \to \omega^2$ be a bijection with coordinate maps $\phi_0$ and $\phi_1$. Inductively we will construct sequences $\{p_n \; | \; n < \omega\}$, $\{\dot{t}_n \; | \; n < \omega\}$, $\{\dot{k}_n \; | \; n < \omega\}$ and a set $X = \{x_s \; | \; s \in \omega^{<\omega}\}$ so that the following hold.
\begin{enumerate}
\item
$p_0 = p$
\item
$\{p_n \; |\; n < \omega\} \subseteq M \cap \P(\mathcal K)$ and for all $n <\omega$ $p_{n+1} \leq p_n$
\item
$X \subseteq 2^\omega \cap M$ is nice for $\mathcal K$
\item
$X \subseteq [p_n]$ for every $n < \omega$
\item
For each $n$ $\dot{t}_n$ is a $\P(\mathcal K)$ name for a node in $\dot{T}_{\phi_0(n)}$ extending the $\phi_1(n)^{\rm th}$-node in $\dot{T}_{\phi_0(n)}$
\item
For each $n < \omega$ $\dot{k}_n$ is a name for an element of $\omega$ in $M$
\item
For each $s \in \omega^n$ and $i, m \in \omega$ if $m = \Delta(x_s, x_{s^\frown i})$ and $t = (x_{s^\frown i}) \hook m$ then $(p_{n+1})_t \in D_n$, forces for some $s_n \in \omega^{<\omega}$ that the $\phi_1(n)^{\rm th}$-node in $\dot{T}_{\phi_0(n)}$ is $\check{s}_n$ and $(p_{n+1})_t \forces \dot{k}_n \in {\rm dom}(\dot{t}_n) \setminus {\rm dom}(\check{s}_n) \land \dot{t}_n(\dot{k}_n) = \check{g} (k)$.
\end{enumerate}

Suppose first that we can construct such a sequence. Let $q = \bigcap_{n < \omega} p_n$. Clearly $q$ is a condition since it's a perfect tree in which $X$ is dense. From this it follows almost immediately that it is $(M, \P(\mathcal K), \calE, g)$-generic as needed.

Therefore it remains to construct the sequence described above. This is done by induction. Let $p_0 = p$, $x_{\emptyset}$ be any element of $[p_0] \cap M$. Now assume that $\{p_j \; | \; j < n+1\}$, $\{x_s \; | \; s \in \omega^{\leq n}\}$, $\{\dot{t}_j \; | \; j < n\}$ and $\{\dot{k}_j \; | \; j < n\}$ have been defined. We will define $p_{n+1}$, $\{x_s \; | \; s \in \omega^{n+1}\}$, $\dot{t}_n$ and $\dot{k}_n$. For each $s \in \omega^n$ let $l \in \omega$ be so that $l > \Delta (x_s, x_{s'})$ for all $s ' \subsetneq s$. Define $Y_s$ to be the set of all $m > l$ so that $x_s \hook m \in p_n$ is a splitting node. For each $m \in Y_s$ let $t_m = (x_s \hook m)^\frown (1 - x_s (m))$ (which is a node of $p_n$ since $x_s \hook m$ is splitting). Let $p^s_m = (p_n)_{t_m}$. Note that $p^s_m \in M$. By strengthening if necessary we may also assume that there is an $u^s_m \in {\omega}^{\omega}$ so that $p^s_m \forces$``$\check{u}^s_m$ is the $\phi_1(n)^{\rm th}$ node in $T_{\phi_0(n)}$". Now let $W^s_m = \{t \in \omega^{<\omega} \; | \; t \supsetneq u^s_m \; {\rm and} \; p^s_m \nVdash t \notin \dot{T}_{\phi_0(n)}\}$. As in the proof of Theorem \ref{preservation}, $W^s_m$, alongside $u^s_m$ and its predecessors is a tree in $\mathcal I_T(\calE)^+$ (in $M$) so $g$ densely diagonalizes it. Let $r^s_m \leq p^s_m$ so that $r^s_m \in D_n$, $[r^s_m] \cap C_{\alpha_z} = \emptyset$ for every $z \subseteq s$ and there is a $t^s_m \in W^s_m$ and a $k^s_m < \omega$ so that $k^s_m \in {\rm dom}(t) \setminus{\rm dom}(u^s_m)$ and $r^s_m \forces \check{t}^s_m \in \dot{T}_{\phi_0(n)}$ and $t^s_m(k^s_m) = g(k^s_m)$. Enumerate $Y_s$ as $\{m^s_i \; | \; i < \omega\}$. For each $i < \omega$ choose $x_{s^\frown i}$ to be any branch in $r^s_{m^s_i}$ and let $p_{n+1} = \bigcup\{r^s_{m^s_i} \; | \; s \in \omega^n \land i \in \omega\}$. Finally let $\dot{t}_n = \{(\check{t}^s_{m^s_i}, r^s_{m^s_i}) \; | \;  s \in \omega^n \land i \in \omega\}$ and $\dot{k}_n = \{(\check{k}^s_{m^s_i}, r^s_{m^s_i}) \; | \;  s \in \omega^n \land i \in \omega\}$. Clearly these suffice.

\end{proof}

As before essentially the same proof gives the analogous result for tight eventually different families of permutations. We state the theorem below and leave the (primarily cosmetic) modifications of the above proof to the reader.

\begin{theorem}
Let $\mathcal K \subseteq P(\cantor)$ be an uncountable partition of $\cantor$ into closed sets and let $\calP$ be a tight eventually different family of permutations. The forcing notion $\P(\mathcal K)$ strongly preserves the tightness of $\calP$.
\end{theorem}

As a result of these theorems we have the following.
\begin{theorem}
In the iterated partition forcing model the inequality $\mfa_e = \mfa_p  = \mfd < \mfa_T$ holds. 
\label{partitionpreservation2}
\end{theorem}

\begin{proof}
Using a bookkeeping device to keep track with the partitions, let $\P$ be the $\omega_2$ length countable support of partition forcing. Let $G\subseteq\P$ be generic. In $V[G]$ we have $\mfa_e = \aleph_1$ by Theorem \ref{partitionpreservation1} and $\mfa_p = \aleph_1$ by Theorem \ref{partitionpreservation2}. Moreover this forcing is known to be $\baire$-bounding (\cite{partitionnumbers}) so $\mfd = \aleph_1$. Finally that $\mfa_T = \aleph_2$ in this model is by \cite[Theorem 6]{MillerCOV} (indeed this is what it was designed to do). 
\end{proof}

\begin{remark}
In \cite{IndependentCompact} it was shown that in the iterated partition forcing model that $\mfa = \mathfrak{i} = \mathfrak{u} = \aleph_1$. Alongside Theorem \ref{partitionpreservation2} this result gives credence to the idea that the iterated partition forcing model is a model where every relative of $\mfa$ other than $\mfa_T$ is $\aleph_1$. It remains open what happens in this model to $\mfa_g$.
\end{remark}

Putting together the results of this section, the previous one, Fact \ref{randommodel} and Proposition \ref{randommodel2} we get the following which encapsulates one of the main results of this paper.

\begin{corollary}
Any $\{\aleph_1, \aleph_2\}$-valued assignment of $\{\mfa_e, \mfa_p, \mfd, \mfa_T\}$ respecting $\mfd \leq \mfa_T$ and $\mfa_e = \mfa_p$ is consistent. In particular, $\mfa_p$ and $\mfa_e$ are both independent of both $\mfd$ and $\mfa_T$.
\label{ind}
\end{corollary}

\section{$h$-Perfect Trees and the consistency of $\mfa_e = \mfa_p = \mathfrak{u} = \mathfrak{d} < non(\Null) = cof(\Null)$}
Recall that for a function $h:\omega \to \omega$ with $1 < h(n) < \omega$ for all $n < \omega$, the forcing notion $\mathbb{PT}_h$, sometimes called $h$-perfect tree forcing, consists of trees $p \subseteq \omega^{<\omega}$ so that the following hold:
\begin{enumerate}
\item
For all $t \in p$ and all $l \in {\rm dom}(t)$ we have $t(l) < h(l)$.
\item
Every $t \in p$ has either one or $h(l(t))$-many immediate successors in $T$. 
\item
For every $t \in p$ there is a $t ' \supseteq t$ with $t' \in p$ and there are $h(l(t'))$ many immediate successors of $t'$ in $p$.
\end{enumerate}

This forcing notion was first considered in \cite{SMZnoCohen}. For simplicity here we will restrict our attention to the case $h(n) = 2^n$ for all $n<\omega$. Obviously much more can be said but for our purposes this is not necessary. In \cite{SMZnoCohen} the following is shown.

\begin{fact}[\cite{SMZnoCohen}]
Denote by $\mathbb{PT}_{2^n}$ the forcing notion $\mathbb{PT}_h$ for $h(n) = 2^n$ for all $n < \omega$. The following hold.
\begin{enumerate}
\item
$\mathbb{PT}_{2^n}$ is proper, and in fact satisfies Axiom A.
\item
$\mathbb{PT}_{2^n}$ is $\baire$-bounding.
\item
$\mathbb{PT}_{2^n}$ preserves $P$-points.
\item
$\mathbb{PT}_{2^n}$ makes the ground model reals measure zero.
\end{enumerate}
\end{fact}

It follows from this that a countable support iteration of $\mathbb{PT}_{2^n}$ over a model of $\CH$ will force $\aleph_1 = \mathfrak{u} = \mathfrak{d} < non(\Null) = cof(\Null) = 2^{\aleph_0} = \aleph_2$. Will will show that $\mathbb{PT}_{2^n}$ strongly preserves the tightness of all tight eventually different families and tight eventually different families of permutations.

\begin{theorem}
Let $\calE$ be a tight eventually different family. The forcing notion $\mathbb{PT}_{2^n}$ strongly preserves the tightness of $\calE$.
\label{PTH1}
\end{theorem}

The proof of this theorem is nearly verbatim to the analogous result for Miller forcing, Theorem \ref{miller}. We include it below for completeness. First we adapt some terminology from other arboreal forcings. If $p \in \mathbb{PT}_{2^n}$ and $n < \omega$ then a node $t \in p$ is an ${n}^{\rm th}${\rm -splitting node} if it has $2^{l(t)}$ many immediate successors and it has the $n - 1$  predecessors with this property. Denote by ${\rm Split}_n(p)$ the set of $n$-splitting nodes. We say that for two $2^n$-perfect trees $p, q \in\mathbb{PT}_{2^n}$ that $q \leq_n p$ if $q \leq p$ and for all $i < n + 1$ ${\rm Split}_i(p) ={\rm Split}_i(q)$. 

\begin{proof}
 Let $p\in \mathbb{PT}_{2^n}$ be a condition, let $M\prec H_\theta$ countable with $\theta$ sufficiently large, $p, \mathbb{PT}_{2^n}, \calE \in M$. Let $g \in \calE$ densely diagonalize every $T \in M \cap \mathcal I_T(\calE)^+$. Let $\{T_n \; | \; n < \omega\}$ be an enumeration of all $\mathbb{PT}_{2^n}$ names in $M$ for trees in $\mathcal I_T(\calE)^+$. Let $\{D_n \; | \; n < \omega\}$ enumerate all dense open subsets of $\mathbb{PT}_{2^n}$ in $M$. Let $\phi:\omega \to \omega^2$ be a bijection with coordinate maps $\phi_0$ and $\phi_1$. Inductively we will construct sequences $\{p_n \; | \; n < \omega\}$, $\{\dot{t}_n \; | \; n < \omega\}$, $\{\dot{k}_n \; | \; n < \omega\}$ so that the following hold.
\begin{enumerate}
\item
$p_0 = p$
\item
$\{p_n \; |\; n < \omega\} \subseteq M$ and for all $n <\omega$ $p_{n+1} \leq_{n+1} p_n$
\item
For each $n$ $\dot{t}_n$ is a $\mathbb{PT}_{2^n}$ name in $M$ for a node in $\dot{T}_{\phi_0(n)}$ extending the $\phi_1(n)^{\rm th}$-node in $\dot{T}_{\phi_0(n)}$
\item
For each $n < \omega$ $\dot{k}_n$ is a name for an element of $\omega$ in $M$
\item
For each ${n+1}^{\rm st}$-splitting node $t$ of $p_{n+1}$ we have that $(p_{n+1})_t \in D_n$, and forces for some $s_n \in \omega^{<\omega}$ that the $\phi_1(n)^{\rm th}$-node in $\dot{T}_{\phi_0(n)}$ is $\check{s}_n$ and $(p_{n+1})_t \forces \dot{k}_n \in {\rm dom}(\dot{t}_n) \setminus {\rm dom}(\check{s}_n) \land \dot{t}_n(\dot{k}_n) = \check{g} (k)$.
\end{enumerate}

Obviously if such a family of sequences can be constructed then, letting $q = \bigcap_{n < \omega} p_n$, it is clear that $q$ is $(M, \mathbb{PT}, \calE, g)$-generic as needed.

Therefore it remains to construct the requisite sequence. This is done by induction. Let $p_0 = p$. Now assume that $\{p_j \; | \; j < n+1\}$, $\{\dot{t}_j \; | \; j < n\}$ and $\{\dot{k}_j \; | \; j < n\}$ have been defined. We will define $p_{n+1}$, $\dot{t}_n$ and $\dot{k}_n$. For each ${n+1}^{\rm st}$-splitting node $t$ of $p_n$, let $q_t \leq (p_n)_t$ be an element of $M \cap D_n$ which forces for some $s_n \in \omega^{<\omega}$ that the $\phi_1(n)^{\rm th}$-node in $\dot{T}_{\phi_0(n)}$ is $\check{s}_n$. Let $W_t = \{u \supsetneq s_n \; | \; q_t \nVdash u \notin \dot{T}_{\phi_0(n)}\}$. As in the proof of Lemma \ref{preservation} $W_t$ alongside $s_n$ and its predecessors is a tree in $\mathcal I_T(\calE)^+ \cap M$ hence it is densely diagonalized by $g$. Therefore we can find a $u_t \in W_t$ and a $k_t \in {\rm dom}(u_t) \setminus {\rm dom}(s_n)$ so that $u_t(k_t) = g(k_t)$. Let $r_t \leq q_t$ force that $\check{u_t} \in \dot{T}_{\phi_0(n)}$. Now let $p_{n+1} = \bigcup_{t \in {\rm Split}_{n+1}(p_n)} r_t$. Let $\dot{t}_n = \{(\check{u}_t, r_t) \; | \; t \in {\rm Split}_{n+1}(p_n)\}$ and $\dot{k}_n =\{(\check{k}_t, r_t) \; | \; t \in {\rm Split}_{n+1}(p_n)\}$. Clearly these suffice.

\end{proof}

As always the analgous result for permutations is proved in an almost identical way.

\begin{theorem}
Let $\calP$ be a tight eventually different family of permutations. The forcing notion $\mathbb{PT}_{2^n}$ strongly preserves the tightness of $\calP$.
\label{PTH2}
\end{theorem}

\begin{remark}
Theorems \ref{PTH1} and \ref{PTH2} hold regardless of which $h$ is chosen. It follows that an iteration of $\mathbb{PT}_h$ forcing notions where many different $h$'s are chosen, including the case of Miller forcing $(h(n) = \omega$ for all $n < \omega)$, as is done in \cite[Theorem 3.9]{SMZnoCohen}, will strongly preserve tight eventually different families and tight eventually different families of permutations. In particular the existence of tight eventually different families of functions and permutations of size $\aleph_1$ are consistent with the additivity of the strong measure zero ideal being $\aleph_2 = 2^{\aleph_0}$, even when no Cohen reals are added.
\end{remark}

Now we can conclude from Theorems \ref{PTH1} and \ref{PTH2} the following.

\begin{theorem}
Let $\P$ be the $\omega_2$-length countable support iteration of $\mathbb{PT}_{2^n}$ and let $G \subseteq \P$ be generic over $V$. In $V[G]$ we have $\mfa_e = \mfa_p = \mfa = \mfd = \mathfrak{u} < non(\Null) = cof(\Null)$. 
\end{theorem}

\section{Shelah's Forcing $\Q_\mathcal I$ and the consistency of $\mfa_e = \mfa_p = \mathfrak{i} < \mathfrak{u}$}
In this section we show that Shelah's forcing $\Q_\mathcal I$ for proving the consistency of $\mathfrak{i} < \mathfrak{u}$ strongly preserves tight eventually different families and tight eventually different sets of permutations. As a result we obtain the consistency of $\mfa = \mfa_e = \mfa_p = \mfd = \mathfrak{i} = cof(\Null) < \mathfrak{u}$. First we recall the poset $\Q_\calI$ from \cite{SS1} and some of its properties. The exposition in this section strongly mirrors that of \cite[Section 5]{IndependentCompact}, where it is shown that $\Q_\calI$ strongly preserves tight MAD families.

%

We start, by recalling some terminology from \cite{SS1}, see also \cite{IndependentCompact}. Given an ideal $\calI$ on $\omega$, we say that 
an equivalence relation $E$ on a subset of $\omega$ is an $\calI$-equivalence relation if ${\rm dom}(E) \in \calI^*$ (the dual filter) and each $E$-equivalence class is in $\calI$. For $\calI$-equivalence relations $E_1$ and $E_2$, we define $E_1 \leq_\calI E_2$ if ${\rm dom}(E_1) \subseteq {\rm dom}(E_2)$ and every $E_1$ equivalence class is the union of a set of $E_2$ equivalence classes. Moreover, we will make use of the notion of a $A$-$n$-determined function. More precisely: Given a subset $A$ of $\omega$, we say that a function  $g:2^A \to 2$ is $A$-$n$-determined  if there is a set $a\subseteq A\cap n+1$ such that whenever  $\eta \hook a = \nu \hook a$ for $\eta,\nu$ in ${^A\omega}$, we have  $g(\eta) = g(\nu)$. For each $i\in A$, $g_i$ is the function mapping $\eta\in {^A 2}$ to $\eta(i)$. The following claim appears in~\cite{SS1}: If $g$ is a $A$-$n$-determined function,  then $g=\varphi(g_0,...,g_n)$, where $\varphi(g_0,...,g_n)$ which is obtained as a maximum, minimum and complement (e.g. $1- g_i$) of $g_0, ..., g_n$ and the constant functions on $0$ and $1$. Again following~\cite{SS1}, given an $\calI$-equivalence relation $E$, we denote by $A = A(E) = \{x \; | \; x \in {\rm dom}(E) \; {\rm and} \; x = {\rm min}[x]_E\}$.

\begin{definition}[The conditions of $\Q_\calI$]
Let $\calI$ be an ideal on $\omega$. Define $\Q_\calI$ to be the set of $p = (H^p, E^p)$ where
\begin{enumerate}
\item
$E^p$ is an $\calI$-equivalence relation,
\item
$H^p$ is a function with domain $\omega$ so that for each $n$ $H(n)$ is $A(E)$-$n$-determined,
\item
if $n \in A(E^p)$ then $H^p(n) = g_n$,
\item
if $n \in {\rm dom}(E^p) \setminus A(E^p)$ and $n E i$ for an $i \in A(E^p)$ then $H^p(n)$ is either $g_i$ or $1-g_i$.
\end{enumerate}
\end{definition}
For a condition $p \in \Q_\calI$ let $A^p = A(E^p)$. Before we can define the order on $\Q_\calI$ we need one more definition, again appearing in~\cite{IndependentCompact}: For $p, q \in \Q_\calI$ with $A^p \subseteq A^q$, we write $H^p(n) =^{**} H^q(n)$ if for each $\eta \in 2^{A^p}$ we have $H^p(n)(\eta) = H^q(n)(\eta ' )$ where

\begin{center}
$\eta ' (j) = 
\begin{cases}

\eta(j) & j \in A^p \\

H^p(j)(\eta) & j \in A^q \setminus A^p \\

\end{cases}
$
	\end{center}

Now we can define the order on $\Q_\calI$.
\begin{definition}
Let $p, q \in \Q_\calI$. We let $p \leq q$ if 
\begin{enumerate}
\item
$E^p \leq_\calI E^q$
\item
if $H^q(n) = g_i$ for $n \in {\rm dom}(E^q)$ then $H^p(n) = H^p(i)$,
\item
if $H^q(n)  = 1 - g_i$ for $n \in {\rm dom}(E^q)$ then $H^p(n) =^{**} H^q(n)$,
\item
if $n \in \omega \setminus {\rm dom}(E^q)$ then $H^p(n) =^{**} H^q(n)$.
\end{enumerate}
Finally let $p \leq_n q$ if $p \leq q$ and $A^p$ contains the first $n$ elements of $A^q$.
\end{definition}

If $\calI$ is a maximal ideal $\calI$ then the forcing notion $\Q_\calI$ is proper \cite[Claim 1.13]{SS1}. Moreover it has the Sacks property \cite[Claim 1.12]{SS1} and hence is $\baire$-bounding. Forcing with $\Q_\calI$ kills the maximality of $\calI$ \cite[Claim 1.5]{SS1}. As a result, using a bookkeeping device to ensure we cover all possible $\calI$'s, iteratively forcing with $\Q_\calI$ makes $\mathfrak{u} = \aleph_2$.

\begin{lemma} 
Let $p \in \Q_\calI$. For an initial segment $u$ of $A^p$ and $h:u \to 2$ let $p^{[h]}$ be the pair $q = (H^q, E^q)$ defined by letting, for each $n \in \omega$ $H^q(n) = \varphi(g_0, ..., g_i/h(i), ..., g(n))$ where $H^p(n) = \varphi(g_0, ..., g_n)$ and the substitution is done only for $i \in u$ and $E^q = E^p \hook \bigcup\{[i]_{E^p} \; | \; i \in A^p\setminus u\}$.
\begin{enumerate}
\item \cite[Claim 1.7, (2)]{SS1}
$p^{[h]}$ is a condition in $\Q_\calI$ extending $p$ and the set of $\{p^{[h]} \; | \; h \in 2^u\}$ is predense below $p$.
\item \cite[Claim 1.8]{SS1}.
If $u$ is a the set of the first $n$ elements of $A^p$ and $D$ is a dense subset of $\Q_\calI$ then there is a $q \in \Q_\calI$ so that $q \leq_n p$ and $q^{[h]} \in D$ for any $h \in 2^u$.
\end{enumerate}
\label{densityQI}
\end{lemma} 

Finally we will need the following game to analyze fusions of conditions from $\Q_\calI$.

\begin{definition}[The Game $\mathsf{GM}_\calI(E)$]
The game $\mathsf{GM}_\calI(E)$ is played as follows. On the $n^{\rm th}$-move, the first player chooses an $\calI$-equivalence relation $E^1_n \leq_\calI E^2_{n-1}$ ($E^1_0 = E$), and the second played chooses an $\calI$-equivalence relation $E^2_n \leq_\calI E^1_n$. After $\omega$ many moves the second player wins if and only if $\bigcup_{n > 0} {\rm dom}(E^2_{n-1}) \setminus {\rm dom}(E^1_n) \in \calI$.
\end{definition}

\begin{remark}
Note that if some play of the game $\mathsf{GM}_\calI(E)$ is given where player II wins playing $\{E^2_n\}_{n < \omega}$ then player II also wins the play where for each $n < \omega$ they play instead some $\{E^{2, *}_n\}_{n < \omega}$ with $E_{n+1}^1 \leq_\mathcal I E^{2, *}_n \leq_\mathcal I E^2_n$. 
\label{game-second}
\end{remark}

\begin{lemma}[Claim 1.10 (1) of \cite{SS1}]
If $\calI$ is a maximal ideal then player I has no winning strategy in the game $\mathsf{GM}_\calI(E)$.
\end{lemma}

Putting all of this together we can now show the following.
\begin{theorem}
For any maximal ideal $\calI$ and any tight eventually different family $\calE$ the forcing notion $\Q_\calI$ strongly preserves the tightness of $\calE$.
\end{theorem}

\begin{proof}
Let $p\in\Q_\calI$, $M$ a countable elementary submodel of $H_\theta$ for $\theta$ sufficiently large such that $\calI, \calE, p\in M$ and let $g\in\calE$ densely diagonalize every $T\in\calI_T(\calE)^+\cap M$. We fix an enumeration $\{D_n \; | \; n\in\omega\}$ of all open dense subsets of~$\Q_\calI$ that are in $M$, and an enumeration $\{\dot{T}_n \; | \; n\in\omega\}$ of all $\Q_\calI$-names for elements of $\calI_T(\calE)^+$ that are in $M$. Let $\phi:\omega \to \omega^2$ with coordinate functions $\phi_0$ and $\phi_1$.
\par
We define a strategy for the first player in the game $\mathsf{GM}_\calI(E)$, which cannot be winning in all rounds. 
\par
We set $p_0=q_0=p$ and $u_0=\emptyset$. We assume that the~first player has chosen~$E^1_n$, $q_n$, $p_n$, $u_n$, and the second one an $E^2_n$. We give instructions to choose $E^1_{n+1}$, $q_{n+1}$, $p_{n+1}$, $u_{n+1}$. We begin with $q_{n+1}$:
\begin{enumerate}
    \item ${\rm dom} (E^{q_{n+1}})={\rm dom} (E^{p_n})$,
    \item $xE^{q_{n+1}}y$ if and only if one of the following holds:
    \begin{enumerate}
        \item 
$xE^2_ny$.
        \item 
There is $k\in u_n$ with $x,y\in[k]_{E^{p_n}}$ and $x,y\not\in{\rm dom} (E^2_n)$.
        \item 
There are $k_0,k_1\not\in\bigcup\{[i]_{E^{p_n}} \; |\; i\in u_n\}$ with $x\in[k_0]_{E^{p_n}}$, $y\in[k_1]_{E^{p_n}}$ and $k_0,k_1\not\in{\rm dom} (E^2_n)$.        
    \end{enumerate}
    \item $H^{q_{n+1}}$ is chosen such that:
    \begin{enumerate}
        \item If $l\in\omega\setminus{\rm dom} E^{p_n}$ then $H^{q_{n+1}}(l)=^{**}H^{p_n}(l)$.
        \item If $l\in{\rm dom} (E^{p_n})\setminus A^{q_{n+1}}$, $H^{p_n}(l)=g_i$ then $H^{q_{n+1}}(l)=H^{q_{n+1}}(i)$.
        \item If $l\in{\rm dom} (E^{p_n})\setminus A^{q_{n+1}}$, $H^{p_n}(l)=1-g_i$ then $H^{q_{n+1}}(l)=1-H^{q_{n+1}}(i)$.
        \item If $l\in A^{p_n}\setminus A^{q_{n+1}}$ then $H^{q_{n+1}}(l)=^{**}H^{p_n}({\rm min}[l]_{E^{q_{n+1}}})$.
    \end{enumerate}
\end{enumerate}
Note that for the already defined condition $q_{n+1}$ we have $q_{n+1}\leq_n p_n$. Take $u_{n+1}=u_n\cup\{\min(A^{q_{n+1}}\setminus u_n)\}$. Let $D'_n$ be the set of all $r \leq p$ so that
\begin{enumerate}
\item
$r$ decides the $\phi_1(n)^{\rm th}$-node of $\dot{T}_{\phi_0(n)}$ to be some $\check{s}_n$

\item
There is a $k < \omega$ and a $t \supseteq s_n$ so that $k \in {\rm dom}(t) \setminus {\rm dom}(s_n)$ and $r \forces \check{t} \in \dot{T}_{\phi_0(n)}$ and $g(k) = t(k)$.
\end{enumerate}

As in the previous proofs, the set $D'_n$ is open dense below $p$ (and also below $q_{n+1}$). Then $D'_n\cap D_n$ is dense below~$q_{n+1}$. Therefore we can apply Lemma \ref{densityQI} to obtain $p_{n+1}\leq_{n+1} q_{n+1}$ such that for each $h\in{}^{u_{n+1}}\{0,1\}$, the condition $p^{[h]}_{n+1} \in D'_n\cap D_n\cap M$. In particular, if $h\in2^{u_{n+1}}$ then $p_{n+1}^{[h]}$ decides the $\phi_1(n)^{\rm th}$ node of $\dot{T}_{\phi_0(n)}$ to be some $s_n$, forces that there is a $k < \omega$ and a $t \supseteq s_n$ so that $k \in {\rm dom}(t) \setminus {\rm dom}(s_n)$ and $g(k) = t(k)$ and $p_{n+1}^{[h]}\in D_n\cap M$. It follows that $p_{n+1}\Vdash$``there is a $t$ in $\dot{T}_{\phi_0(n)}$ extending the $\phi_1(n)^{\rm th}$ node and a $k \in {\rm dom}(t)$ above the domain of the $\phi_1(n)^{\rm th}$ node of $\dot{T}_{\phi_0(n)}$ so that $t(k) = g(k)$". Finally, we set

\begin{equation*}
E^1_{n+1}=E^{p_{n+1}}\hook({\rm dom} (E^{p_{n+1}})\setminus\bigcup\{[i]_{E^{p_{n+1}}}\; | \; i\in u_{n+1}\}).
\end{equation*}

We define a fusion $q$ of a sequence $\langle p_n \; | \; n\in\omega\rangle$. Relation~$E^q$ has ${\rm dom} (E^q)=\bigcap\{{\rm dom} (E^{p_n}) \; | \; n\in\omega\}$, and $xE^qy$ if for every~$n$ large enough, $xE^{p_n}y$. Function $H^q$ is equal to $H^{p_n}$ for large enough $n$. {\bf If} $q \in \Q_\mathcal I$ then clearly $q$ is an $(M, \Q_\mathcal I)$-generic condition and $q \leq_n p_n$ for all $n < \omega$ so $q$ forces that $g$ densely diagonalizes every $\dot{T}_n$ as needed. However it's not obvious that $q$ is a condition, specifically it's not immediate that ${\rm dom}(E^q) \in \mathcal I^*$. Indeed this does not necessarily happen but it does if player II wins. Since $\mathsf{GM}_\mathcal I(E)$ is not determined for player I let us choose a play of the game where the first player uses the described strategy above, but still looses. Thus the second player wins. To complete the proof we need to show in this case that ${\rm dom}(E^q) \in \mathcal I^*$ and hence $q$ is a condition. This will follow from the following sequence of claims. To be clear, since player II won the game we have $\bigcup_{n > 0} {\rm dom} (E^2_{n-1}) \setminus {\rm dom} (E^1_n )\in \mathcal I$.

\begin{claim}
We can assume that ${\rm min} ({\rm dom}( E^2_n))>\max u_{n+1}$.
\end{claim}

\begin{proof}
By remark \ref{game-second} if the $u_{n+1}$'s were known at the stage where player II played $E^2_n$ then this would be fine however there is an apparent issue of primacy here: $u_{n+1}$ is chosen after $E^2_n$. However, note that $u_{n+1}$ is read off from $E^{q_{n+1}}$ which is read off from $E^{p_n}$ and $E^2_n$ so player 2 can choose $E^2_n$ so as to affect $u_{n+1}$. 

It remains to show that some choice of $E^2_n$ will result in $\max u_{n+1} < {\rm min} ({\rm dom}(E^2_n))$. This is argued as follows. In words, the first step of player I's strategy at stage $n+1$, when $q_{n+1}$ is built, is to build an equivalence relation on ${\rm dom} ( E^{p_n})$ so that $x$ and $y$ are related just in case either 
\begin{enumerate}
\item
they are $E^2_n$ related (note that ${\rm dom}( E^2_n) \subseteq {\rm dom}( E^1_n) \subseteq (E^{p_n})$), 
\item
they are $E^{p_n}$ related and in one of the first $n$ equivalence classes of $E^{p_n}$ or 
\item
they fall into neither of these categories i.e. everything else is put together. 
\end{enumerate}
In other words, $q_{n+1}$ forgets everything in $E^{p_n}$ outside of $E^2_n$ except for the first $n$ equivalence classes. Note that $E^2_n$ will not contain anything from those first $n$ equivalence classes since they were thrown out of $E^1_n$ whose domain is a super set of $E^2_n$'s domain. Therefore the $n+1^{\rm st}$ equivalence class of $q_{n+1}$ will either be an equivalence class from $E^2_n$ or this ``everything else" class. By thinning out enough player 2 can ensure that it is the ``everything else" class by simply making the minimum of ${\rm dom}( E^2_n)$ greater than the first thing in ${\rm dom}(E^{p_n})$ not in $\bigcup \{[k]_{E^{p_n}} \; | \; k \in u_n\}$ (which must exist since ${\rm dom}(E^{p_n}) \notin \mathcal I$ but $\bigcup \{[k]_{E^{p_n}} \; | \; k \in u_n\}$ is the union of finitely many elements of $\mathcal I$ and hence in $\mathcal I$).
\end{proof}

\begin{claim}
If ${\rm min} ({\rm dom} (E^2_n))>\max u_{n+1}$ then ${\rm dom}( E^{p_n})\setminus {\rm dom}( E^2_n)\subseteq\bigcup\{[k]_{E^{q_{n+1}}} \; | \; k\in u_{n+1}\}$.
\end{claim}

\begin{proof}
By what we said in the last proof, if ${\rm min} ({\rm dom} (E^2_n))>\max u_{n+1}$ then the new equivalence class in $u_{n+1}$ is the ``everything else" class so the only thing that is left in ${\rm dom} (E^{p_n})$ after throwing this out, alongside its first $n$ equivalence classes is whatever was in ${\rm dom} (E^2_n)$.
\end{proof}

\begin{claim}
If ${\rm dom} (E^{p_n})\setminus {\rm dom} (E^2_n)\subseteq\bigcup\{[k]_{E^{q_{n+1}}} \; | \; k\in u_{n+1}\}$, then $\bigcap \{ {\rm dom} (E^{p_n}) \; | \; n < \omega\} \in  \mathcal I^*$. 
\end{claim}

\begin{proof}
Recall that our assumption is that $\bigcup_{n > 0} {\rm dom} (E^2_{n-1} )\setminus {\rm dom}( E^1_n )\in \mathcal I$. It follows that it is enough to show that ${\rm dom} (E_0) \setminus (\bigcup_{n > 0} ({\rm dom} (E^2_{n-1}) \setminus {\rm dom} (E^1_n)))$ is contained in $\bigcap \{ {\rm dom} (E^{p_n}) \; | \; n < \omega\}$ where $E_0$ is the equivalence relation for the condition $p_0$. 

For readability let us label $A:= {\rm dom}( E_0) \setminus (\bigcup_{n > 0} ({\rm dom} (E^2_{n-1}) \setminus {\rm dom} (E^1_n)))$. For all $x \in {\rm dom} (E_0)$ we have $x \in A$ if and only if for each $0 < n < \omega$ (either $x \in {\rm dom}( E^1_n)$ or else $x \notin {\rm dom} (E^2_{n-1})$) (recall that ${\rm dom} (E_n^1) \subseteq {\rm dom} (E^2_{n-1})$).

Note that if $x \notin {\rm dom} (E^2_{n-1})$ for some $n$ then for all $k \geq n$ we have $x \notin {\rm dom} (E^1_k)$ so the above becomes: $x \in A$ if and only if for all $0 < n < \omega$ $x \in {\rm dom} (E^1_n)$ or else for all but finitely many $n$ we have $x \notin {\rm dom} (E^2_{n-1})$ with these two options mutually exclusive.

Fix $x \in A$ we will show that $x \in {\rm dom}( E^{p_n})$ for every $n < \omega$. There are two cases to consider.

\noindent \underline{Case 1:} $x \in {\rm dom}( E^1_n)$ for all $n < \omega$. Note that in this case $x \in {\rm dom}( E^{p_n})$ as needed since for all $n < \omega$ we have ${\rm dom}( E^1_n) \subseteq {\rm dom}( E^{p_n})$ by the way the strategy was defined. 

\noindent \underline{Case 2:} $x \notin {\rm dom} (E^2_{n-1})$ for all but finitely many $n$. We prove by induction on $n$ that $x \in {\rm dom} (E^{p_n})$ for all $n < \omega$. By above, plus the fact that ${\rm dom}( E^2_n) \subseteq {\rm dom} (E^1_n)$, there's no problem with the induction up to the least $k$ with $x \notin {\rm dom}( E^2_{k})$. Thus let us suppose that $x \notin {\rm dom}( E^2_{n+1})$ and show that this implies that $x \in {\rm dom} (E^{p_{n+1}})$ under the inductive assumption that $x \in {\rm dom}( E^{p_n})$. Since $x \in {\rm dom} (E^{p_n})$ we have that $x \in {\rm dom} (E^{q_{n+1}})$ since these two domains are the same. More over, since $x \notin {\rm dom} (E^2_n)$ we have that the $E^{q_{n+1}}$ class that $x$ falls into is one of the $\{[k]_{E^{q_{n+1}}} \; | \; k \in u_{n+1}\}$ since ${\rm dom} (E^{p_n})\setminus {\rm dom} (E^2_n)\subseteq\bigcup\{[k]_{E^{q_{n+1}}} \; | \; k\in u_{n+1}\}$. But now by construction we have $p_{n+1} \leq_{n+1} q_{n+1}$ so the first $n+1$ $E^{q_{n+1}}$ classes are the same as the first $n+1$ $E^{p_{n+1}}$ classes and hence $x \in {\rm dom}(E^{p_{n+1}})$ as needed. This induction completes the proof of the claim.

\end{proof}

 With the proof of these three claims we are done with the proof of the theorem.
\end{proof}

As always we have the same result for families of permutations.
\begin{theorem}
For any maximal ideal $\calI$ and any tight eventually different set of permutations $\calP$ the forcing notion $\Q_\calI$ strongly preserves the tightness of $\calP$.
\end{theorem}

Putting all of this together with the fact that $\Q_\calI$ has the Sacks property and iterating $\Q_\calI$ with countable support $\omega_2$ many times alongside some bookkeeping device to handle the ideals $\calI$ forces $\mathfrak{i} = \mfa < \mathfrak{u}$ we get the following.
\begin{theorem}
In the iterated $\Q_\calI$ model we have $\mfa= \mfa_e = \mfa_p = \mathfrak{i} = cof(\Null) < \mathfrak{u}$.
\end{theorem}

\section{Definability}
In this section we show that if $V= L$ there is a $\Pi^1_1$ tight eventually different family and a $\Pi^1_1$ tight eventually set of permutations. The construction is extremely similar to Miller's seminal proof that there is a $\Pi^1_1$ MAD family in $L$, \cite[Theorem 8.23]{Millerpi11}. From our construction we are able to conclude that in all of the models considered above if the ground model is $L$ then there is are $\Pi^1_1$ tight witnesses to $\mfa_e = \aleph_1$ and $\mfa_p = \aleph_1$. This strengthens the main result of \cite{SacksMedf} where it is shown that if $V=L$ then there is a $\Pi^1_1$ Sacks indestructible maximal eventually different family. 

\begin{theorem}
If $V=L$ then there is a $\Pi^1_1$ tight eventually different family.
\label{pi11medf}
\end{theorem}

As noted above the proof of this theorem is extremely similar to \cite[Theorem 8.23]{Millerpi11}. In our opinion this highlights part of the appeal of tightness for sets of functions: it provides a uniform, relatively simple construction of a $\Pi^1_1$ maximal eventually different family which remains maximal (and $\Pi^1_1$) in an extension of $L$ by a large number of posets.

As in \cite[Lemma 8.24]{Millerpi11} the proof relies on the following coding lemma. For this section for each $n < \omega$ let $c_n \in \baire$ be the function with constant value $n$.

\begin{lemma}
Let $\calE_0$ be a countable, eventually different family which contains each $c_n$. Let $\{T_n \; | \; n < \omega\}$ be a countable collection of trees in $\mathcal I_T(\calE_0)^+$. Let $z \in [\omega]^\omega$ be arbitrary. There is a function $f \in \baire$ which is eventually different from each $g \in \calE_0$, densely diagonalizes $T_n$ for each $n$ and computes $z$. Moreover such an $f$ can be found computably from $\calE_0$, $\{T_n \; | \; n < \omega\}$ and $z$.
\label{codinglemma}
\end{lemma}

\begin{proof}
Fix $\calE_0$, $\{T_n\; | \; n < \omega\}$ and $z \in [\omega]^{\omega}$ as in the statement of the lemma. Enumerate $\calE_0 = \{g_k \; | \; k < \omega\}$. Fix a bijection $\phi:\omega \to \omega^2$ with coordinate functions $\phi_0$ and $\phi_1$. We define $f$ in stages. Namely we define a $\subseteq$-increasing sequence $\langle f_i \; | \; i < \omega\}$ so that the following hold. 
\begin{enumerate}
\item
For all $i < \omega$ $f_i$ is a finite partial function from $\omega$ to $\omega$ so that $i \in {\rm dom}(f_i)$.
\item
For all $i < \omega$ there is an $s \in T_{\phi_0(i)}$ extending the $\phi_1(i)^{\rm th}$ node of $T_{\phi_0(i)}$, call it $t_i$ and a $k \in {\rm dom}(s) \setminus {\rm dom}(t_i)$ so that $f_{i+1} (k) =s(k)$ (and in particular is defined).
\item
If $i < j < \omega$ then $f_i \supseteq f_j \cap g_i$. 
\item
For every $n < \omega$ for all $k > n$ $n \notin {\rm range}(f_k)$ and $n \in z$ if and only if $|c_n \cap f_n|$ is even.
\end{enumerate}
Assuming we can define such a sequence, we let $f = \bigcup_{i < \omega} f_i$. This $f$ is as needed since, by 1 it is a function, by 2 it densely diagonalizes every $T_n$, by 3 it is eventually different from every element of $\calE_0$ and by $4$ it computes $z$.

We define $\{f_i \; | \; i < \omega\}$ by recursion as follows. At stage $0$, let $f_0 = \{(0, 0) \}$ if $0 \notin z$ and let $f_0 = \{(0, 0), (1, 0)\}$ if $0 \in z$. Now suppose we have defined $f_n$. Let $k$ be the least number not in the set $\{g_0(n+1), ..., g_n(n+1), 0, ..., n\}$. First, if $n+1 \notin{\rm dom}(f_n)$ let $f_n' = f_n \cup \{(n+1, k)\}$. Now, let $t_{n}$ be the $\phi_1(n)^{\rm st}$ node of $T_{\phi_0(n)}$. Let $s_n$ be least in the lexicographic ordering extending $t_n$ in $T_{\phi_0(n)}$ so that there is a $k \in {\rm dom}(s_n) \setminus {\rm dom}(t_n)$ so that $k \notin {\rm dom}(f_n')$, and $s_n(k) \notin \{g_0(n+1), ..., g_n(n+1), 0, ..., n\}$. That such an $s_n$ and $k$ exists follows by the fact that $T_{\phi_0(n)} \in \mathcal I_T(\calE_0)^+$. Let $f_n '' = f_n ' \cup \{(k, s_n(k))\}$. If $n + 1 \in z$ and $|c_{n+1} \cap f''_n|$ is even or $n + 1 \notin z$ and $|c_{n+1} \cap f''_n|$ is odd then we let $f_{n+1} = f_n ''$. Otherwise, let $l$ be the least element so that $l \notin{\rm dom}(f''_n)$ and $g_i(l) \neq n+1$ for all $i < n+1$. That such an $l$ exists follow from the fact that $\calE_0$ is an eventually different family containing every constant function. Now let $f_{n+1} = f_n'' \cup \{(l, n+1)\}$. This completes the construction and hence the lemma.
\end{proof}

The proof of Theorem \ref{pi11medf} follows along the lines described in \cite[p.194]{Millerpi11} given this coding lemma. We give the details below for completeness.

\begin{proof}[Proof of Theorem \ref{pi11medf}]
Assume $V=L$. We will build a tight eventually different family $\calE = \{f_\alpha \; | \; \alpha < \omega_1\}$ recursively. For each finite $n < \omega$ let $f_n = c_n$. Now suppose we have constructed $\calE_\alpha : = \{f_\xi \; | \; \xi < \alpha\}$  for some $\alpha < \omega_1$. Let $\vec{T}_\alpha$ be the $\leq_L$-least countable sequence of trees in $\mathcal I_T(\calE_\alpha)^+$. We construct $f_{\alpha+1}$ as in the coding lemma with $\calE_\alpha$ our countable eventually different family, $\vec{T}_\alpha$ our sequence of trees and $z$ a real coding an extensional, wellfounded relation $E \subseteq \omega \times \omega$ and $\alpha$ so that $(\omega, E) \cong L_\alpha$. This completes the construction. It just remains to check that it works.

It's not hard to show that the family $\calE$ is tight. To see this, suppose $\calE$ were not tight and let $\vec{T}$ be the $\leq_L$-least countable sequence of trees in $\mathcal I_T(\calE)^+$ which are not densely diagonalized by any element of $\calE$. Let $M \prec L_{\omega_2}$ be a countable elementary submodel containing $\calE$ and $\vec{T}$. Let $L_\delta$ be its transitive collapse. It follows that $L_\delta \models$``$\vec{T}$ is the $\leq_L$-least counter example to tightness of $\calE$" which means that at stage $(\omega_1)^{L_\delta}$ we added a real to $\calE$ densely diagonalizing $\vec{T}$, contradiction. 

What is less clear is that $\calE$ is $\Pi^1_1$. To see this observe that $X \in \calE$ if and only if $X$ codes an extensional, wellfounded relation $E\subseteq \omega \times \omega$ and a ordinal $\alpha$ so that $(\omega, E) \cong L_{\alpha}$ in the way described in the coding lemma and $L_{\alpha+\omega} \models X \in \calE$ where the expression ``$X \in \calE$" is shorthand for the definition of $\calE$ given recursively above (relativized to $L_{\alpha+\omega}$). For more details on this argument the reader is referred to \cite{Millerpi11}.
\end{proof}
Note that in contrast to the case of MAD families this is not the best possible complexity for a maximal eventually different family since $\mathsf{ZF}$ proves that there is a Borel maximal eventually different family (\cite{BorelMedf}) and in fact even a closed one (\cite{CompactnessofMedf}). However, we will apply this result to obtain models where the continuum hypothesis fails and there is still a $\Pi^1_1$ maximal eventually different family of size $\aleph_1$, which will be the best possible complexity for that cardinality. Moreover, in light of Theorem \ref{definable}, this is the best possible complexity for a tight maximal eventually different family (of functions or permutations).

The same argument works in the case of tight eventually different sets of permutations and tight eventually different families. The only point that is different is that in place of the $c_n$'s we need a computable set of disjoint permutations covering $\omega_2$ (in the case of permutations) or an infinite computable partition of $\omega$ into infinite sets. 
\begin{theorem}
If $V=L$ there are a $\Pi^1_1$ tight eventually different set of permutations and a tight MAD family.
\label{pi11other}
\end{theorem}

As a result of Theorems \ref{pi11medf} and \ref{pi11other} we have the following.
\begin{theorem}
The following constellations of cardinal characteristics are compatible with the existence of a $\Pi^1_1$ tight maximal eventually different family of functions, a $\Pi^1_1$ tight eventually different family of permutations and a $\Pi^1_1$ tight MAD family.
\begin{enumerate}
\item
$\mfa= \mfa_e = \mfa_p = \mfd = \mfa_T = \mathfrak{u} < 2^{\aleph_0}$. In this case we can also arrange that the witness to $\mathfrak{u}$ is $\Pi^1_1$.
\item
$\mfa = \mfa_e = \mfa_p = \mathfrak{u} < \mfd = \mfa_T = 2^{\aleph_0}$. In this case we can also arrange that the witness to $\mathfrak{u}$ is $\Pi^1_1$.
\item
$\mfa= \mfa_e = \mfa_p = \mfd < \mfa_T = 2^{\aleph_0}$
\item
$\mfa = \mfa_e = \mfa_p = \mfd = \mathfrak{u} < non(\Null) = cof(\Null) = 2^{\aleph_0}$. In this case we can also arrange that the witness to $\mathfrak{u}$ is $\Pi^1_1$.
\item
$\mfa = \mfa_e = \mfa_p = \mathfrak{i} = cof(\Null) < \mathfrak{u}$. In this case we can also arrange that the witness to $\mathfrak{i} = \aleph_1$.
\end{enumerate}
\end{theorem}

\begin{proof}
For each case we simply add $\omega_2$ many reals over $L$ with countable support of a given type. For item 1 we use Sacks forcing, for item 2 we use Miller forcing, for item 3 we use partition forcing, for item 4 we use $\mathbb{PT}_{2^n}$ and for item 5 we use $\Q_\calI$. For the second to final model note that there is a $\Pi^1_1$ basis for a P-point in $L$ (\cite{Schilhanultrafilter}). For the final model note that there is a $\Pi^1_1$ maximal independent family in $L$ which is preserved by countable support iterations of the form $\P(\mathcal K)$, see \cite{DefnMIFs}.
\end{proof}

\section{Concluding remarks}

The results of the previous sections show that tightness for eventually different family and an eventually different set of permutations is a robust notion and can provide a lot of inside into studying these types of maximal sets of reals in much the same way as tight MAD families provide insight into questions about $\mfa$.

A natural line of inquiry is whether we can define natural notions of tightness for other relatives of $\mfa$. An obvious question in this respect is the following.

\begin{question}
Is there an analogy of tightness for $\mfa_T$ or for $\mfa_g$?
\end{question}

A positive answer to this problem for $\mfa_g$ would presumably strengthen the main result of \cite{indestmcg} in much the same way that Theorem \ref{pi11medf} strengthens the main result of \cite{SacksMedf}.

\bigskip

\noindent {\em Acknowledgments}. The authors would like to thank J\"{o}rg Brendle and Martin Goldstern for some very helpful comments on earlier drafts.



\end{document}